\documentclass[a4paper,11pt]{amsart}
\usepackage{amssymb,amsfonts,amsxtra,
mathrsfs,placeins,graphicx,verbatim,stmaryrd}
\usepackage[all]{xy}
\xyoption{line}
\usepackage{fullpage}
\usepackage{euscript}
\newtheorem{theorem}{Theorem}[section]
\newtheorem{cor}[theorem]{Corollary}
\newtheorem{lem}[theorem]{Lemma}
\newtheorem{prop}[theorem]{Proposition}

\newtheorem{example}[theorem]{Example}
\newtheorem{defi}[theorem]{Definition}
\newtheorem{rem}[theorem]{Remark}

\numberwithin{equation}{section}
\DeclareMathOperator{\Hom}{Hom}
\DeclareMathOperator{\Fib}{Fib}
\DeclareMathOperator{\MC}{MC}
\DeclareMathOperator{\Haut}{HAut}
\DeclareMathOperator{\Aut}{Aut}
\DeclareMathOperator{\Baut}{BAut}
\DeclareMathOperator{\Def}{Def}
\DeclareMathOperator{\defo}{def}

\DeclareMathOperator{\Ker}{Ker}

\newcommand{\noproof}{\begin{flushright} \ensuremath{\square}
\end{flushright}}

\def\ground{\mathbf k}
\def\C{\mathscr C}
\def\L{\mathscr L}

\def\f{\mathfrak f}
\def\g{\mathfrak g}
\def\h{\mathfrak h}
\def\MCmod{\mathscr {MC}}

\def\Ext{\operatorname{Ext}}
\def\Z{\mathbb{Z}}

\def\id{\operatorname{id}}
\def\Der{\operatorname{Der}}
\def\CE{\operatorname{CE}}
\DeclareMathOperator{\Harr}{Harr}
\DeclareMathOperator{\Hoch}{Hoch}

\begin{document}

\title[Models for classifying spaces]{Models for classifying spaces and derived deformation theory.}
\author{A.~Lazarev}
\thanks{The author is grateful to J. Chuang, M. Markl and J. Stasheff for many useful discussions concerning the subject of this paper.}
\thanks{This research was partially supported by the EPSRC grant
  EP/J008451/1.}
\address{University of Lancaster\\ Department of
Mathematics and Statistics\\Lancaster LA1 4YF, UK.}
\email{a.lazarev@lancaster.ac.uk} 
\begin{abstract}
Using the theory of extensions of $L_\infty$ algebras, we construct rational homotopy models for classifying spaces of fibrations, giving answers in terms of classical homological functors, namely the Chevalley-Eilenberg and Harrison cohomology. We also investigate the algebraic structure
of the Chevalley-Eilenberg complexes of $L_\infty$ algebras and show that they possess, along with the Gerstenhaber bracket, an $L_\infty$ structure that is homotopy abelian.
\end{abstract}

\maketitle
\tableofcontents
\section{Introduction}
The problem of classifying fibrations with a given fiber $F$ up to fiber homotopy equivalence was solved by Stasheff, \cite{Sta} (cf. also \cite{All, May})  who proved that
the corresponding functor is represented by the Dold-Lashof classifying space $\Baut(F)$ of the monoid of self-homotopy equivalences of $F$; the homotopy type of the space $\Baut(F)$ depends on the homotopy type of $F$ only. On the other hand we know that the category of rational nilpotent spaces admits a completely algebraic description in terms of either commutative differential graded algebras or differential graded Lie algebras \cite{Q',Sul,BG,N}. It is, therefore, natural to ask for a purely algebraic construction of the space $\Baut(F)$ from a rational homotopy model of $F$. This problem was first addressed in \cite{Sul} and then given more detailed treatment in \cite{SS, Tanre}, in the case of a simply-connected $F$. The answers given in the mentioned references appeared somewhat ad hoc, in that they seemingly depended on the model for $F$ being used and were not expressed in terms of standard derived functors. Note also, that the space $\Baut(F)$ does not completely fall into the realm of rational homotopy since its fundamental group, which is the group of homotopy classes of self-homotopy equivalences of $F$,  is typically not nilpotent. However, it is often the case (for a rational space $F$) that this group is the group of $\mathbb Q$-points of an affine algebraic group and the corresponding Lie algebra was explicitly identified in \cite{BL} as the first Harrison cohomology of the Sullivan model of $F$.

One of the purposes of the present paper is to complete the picture by producing a Lie-Quillen model for the universal cover of the space $\Baut(F)$ together with the action of the fundamental group (or the tangent action of the corresponding Lie algebra) solely in terms of standard derived functors i.e. Chevalley-Eilenberg or Harrison cohomology.

The main tool for us is the notion of an $L_\infty$ extension, first introduced in \cite{CL'}. It is a far reaching generalization of an extension of Lie algebras. It turns out that $L_\infty$ extensions can be classified in terms of Chevalley-Eilenberg cohomology; the corresponding classical result
can be recovered as a (very) special case. Namely, we prove that the functor associating to an $L_\infty$ algebra $V$ the set of equivalence classes of extensions of $V$ with kernel $I$ is represented in the homotopy category of $L_\infty$ algebras by $\Sigma C_{\CE}(I,I)$, the suspended Chevalley-Eilenberg complex of $I$ supplied with the Gerstenhaber bracket.
The topological counterpart of an $L_\infty$ extension is a fibration and taking $I$ to be the $L_\infty$ algebra corresponding to the Sullivan model of $F$ we obtain an explicit model for $\Baut(F)$.

One interesting feature of the notion of an $L_\infty$ extension is that, from a different point of view, this is nothing but a \emph{deformation} with a differential graded base.
This extended, or derived, deformation theory has been the subject of much recent work, cf. \cite{Man} and references therein.
The deformation-theoretic interpretation fits very naturally with the topological notion of a fibration as a family of spaces parametrized by a base. It has to be said that this idea (as well as the realization that rational homotopy theory and deformation theory are, roughly, one and the same thing) belongs to Schlessinger and Stasheff and the present paper obviously owes an intellectual debt to their manuscript, which has been circulated since early 1980's and is now available on arXiv, \cite{SS}. A version of our classification result can be formulated as saying that the deformation functor of an $L_\infty$ algebra $V$ is represented by the truncated Chevalley-Eilenberg complex of $V$. Here our deformation functor really needs to be understood in the derived sense since otherwise it is hardly ever representable. We note that a similar result also holds for other types of homotopy algebras (i.e. $C_\infty$ or $A_\infty$), however the interpretation in terms of extensions is special to $L_\infty$ algebras.

Another application of deformation theory is concerned with the algebraic structure of the Chevalley-Eilenberg complex $C_{\CE}(V,V)$ of an $L_\infty$ algebra $V$. Recall from \cite{Laz} that if $X$ and $Y$ are two rational spaces,  $L(X), L(Y)$ are their Lie-Quillen models and $f:X\to Y$ is a map then, roughly speaking, the Chevalley-Eilenberg complex $C_{\CE}(L(X),L(Y))$ is a Lie-Quillen model for the function space $F(X,Y)$ based at $f$ (more precisely, it is $C_{\CE}(L(X),L(Y))\langle 0\rangle$, the connected cover of the Chevalley-Eilenberg complex). Here the graded Lie
algebra structure on $C_{\CE}(L(X),L(Y))$ comes from the argument $L(Y)$.

Let $X=Y$; then the complex $C_{\CE}(L(X),L(X))$ has two graded Lie algebra structures; one of them is odd,
the other even. The odd Lie algebra corresponds to a Lie-Quillen model for $\Baut(X)$ and the even one -- to a Lie-Quillen model for $\Aut(X)$.  This behavior persists in the abstract situation: the Chevalley-Eilenberg complex $C_{\CE}(V,V)$ of an arbitrary $L_\infty$
algebra $V$ (in particular, of an ordinary Lie algebra) has two graded Lie algebra structures. Note that this phenomenon is more familiar in the context of Hochschild cohomology of an associative (or $A_\infty$) algebra which too, has two structures which combine into a Gerstenhaber algebra structure.

Invoking again the analogy with topology, we expect that one of the Lie algebra structures on $C_{\CE}(V,V)$ (namely, the one coming from the second argument) should be trivial since it corresponds to the Whitehead Lie algebra of a topological monoid. We prove that this is indeed, the case. In fact, we prove that for an arbitrary $L_\infty$ algebra $V$ (not necessarily related to any rational space) the corresponding dgla $C_{\CE}(V,V)$ is \emph{homotopy abelian}, (in the strong sense, i.e. it is $L_\infty$ quasi-isomorphic to an abelian $L_\infty$ algebra). Note that a similar result also holds in the associative context; it could be viewed as a weak version of the Deligne conjecture. To be sure, this version is much easier to prove than the real Deligne conjecture.

The organization of the paper is as follows. In the rest of the introduction we describe our conventions and establish the basic setup in which subsequent work is done. Section 2 describes Chevalley-Eilenberg and Harrison cohomology with coefficients in adjoint representations; this material is not quite standard and the reader may consult \cite{HL} for background and details. The material from this section closely mirrors Section 7 of \cite{Laz} but there is an important difference. The Chevalley-Eilenberg complex of a Lie (or $L_\infty$) algebra with coefficients in itself possesses a Gerstenhaber bracket, and to see that, this complex  is best described in terms of derivations of its representing differential graded algebra; this is the approach adopted here. By contrast, if the coefficients are taken in another Lie (or $L_\infty$) algebra then there exists \emph{another} Lie (or $L_\infty$) structure on the Chevalley-Eilenberg complex, and to see it, one has to use another definition via Maurer-Cartan twisting. It would be interesting to find out the appropriate compatibility relation between such a pair of Lie brackets together with the operad governing them.

Sections 3 and 4 study $L_\infty$ extensions and their variants and links with deformation theory. A classification theorem is established, which could also be viewed as the existence of universal deformations for $L_\infty$ algebras. Here, universality is referred to the corresponding functor being representable in the \emph{homotopy category}, rather than on the nose. Section 5 contains an application of the developed theory to constructing rational homotopy models of classifying spaces. In Section 5 we return to the general algebraic setup and prove that the second $L_\infty$ structure in the Chevalley-Eilenberg complex of $L_\infty$ algebras with coefficients in themselves is homotopy abelian. Finally, in Section 6 we list some open problems naturally arising from our results.

\subsection{Notation and conventions} We make extensive use of the previous work by the author \cite{Laz} and our notation is mostly chosen to coincide with that in op.cit.  We work in the category of $\Z$-graded vector spaces over a field $\ground$
of characteristic zero; any explicit mention of  $\ground$ is usually omitted. When considering models for topological spaces, the field $\ground$ is understood to be $\mathbb Q$. Differential graded algebras will have cohomological grading with upper indices and differential graded Lie algebras will have homological grading with lower indices, unless indicated otherwise. There are cases when the same object has to be considered as homologically graded in one situation and cohomologically graded in another; for this we adopt the standard convention for passing between upper and lower indices: $V_i=V^{-i}$.   The \emph{suspension}  $\Sigma V$ of a homologically graded vector space $V$ is defined by the convention $\Sigma V_i=V_{i-1}$; for a cohomologically graded space the grading convention is as follows: $\Sigma V^i= V^{i+1}$. The functor of taking the linear dual takes homologically graded vector spaces into cohomologically graded ones so that $(V^*)^i=(V_{i})^*$; further we will write $\Sigma V^*$ for $\Sigma(V^*)$; with this convention there is an isomorphism $(\Sigma V)^*\cong\Sigma^{-1}V^*$.

The adjective `differential graded' will be abbreviated as `dg'. A
(commutative) differential graded (Lie) algebra will be abbreviated
as (c)dg(l)a. A Maurer-Cartan (MC) element in a dgla $\g$ is an element $\xi\in\g_{-1}$ which satisfies the MC equation $d(\xi)+\frac{1}{2}[\xi,\xi]=0$. The set of all MC elements in $\g$ is denoted by $\MC(\g)$ and the set of homotopy classes of MC elements in $\g$ is denoted by $\MCmod(\g)$. We will denote by $\g^\xi$ the graded Lie algebra $\g$ supplied with the differential twisted by the MC element $\xi:d^\xi:=d+[\xi,?]$.

A \emph{complete}\footnote{In earlier papers, e.g. \cite{Laz, HL} such dg spaces, as well as accompanying algebraic structures, are called \emph{formal}, which was meant to elicit the analogy with formal power series. Since the term  `formal' has an established, and different, meaning in rational homotopy theory, we have opted for a terminological change in the present paper.} dg vector space is an inverse limit of
finite-dimensional dg vector spaces.   All of our unmarked tensors are understood to be taken over
$\ground$. The tensor product $V\otimes W$ of two complete spaces is understood to be the completed tensor product (and so, it will again be complete). If $V$ is a discrete vector space and $W=\lim_\leftarrow{W_i}$ is
a complete space we will write $V\otimes W$ for
$\lim_{\leftarrow}V\otimes W_i$; thus for two discrete spaces $V$ and
$U$ we have $\Hom(V,U)\cong U\otimes V^*$.

A complete (non-unital) cdga is an inverse limit of finite-dimensional non-unital nilpotent cdgas; the category of non-unital cdgas is equivalent to the category of cocomplete cocommutative coalgebras. For a non-unital cdga $A$ we denote by $\tilde{A}$ the cdga obtained from $A$ by adjoining a unit; such an object will be called a \emph{unital} complete cdga. Given a unital complete cdga $\tilde{A}$, the augmentation ideal of $\tilde{A}$ recovers $A$. We will often omit the adjectives `unital' and `non-unital' when the correct meaning is clear from the context.

We use Hinich's results \cite{H}, on the closed model category of coalgebras, as they are interpreted in \cite{Laz} and we adopt the notation of the latter paper. In particular, we will denote by $\C$ and $\L$ the functors between  the categories of dglas and complete cdgas establishing an equivalence of the corresponding homotopy categories.

For two topological spaces $X$ and $Y$ we will write $[X,Y]$ for the set of homotopy classes of maps $X\to Y$; if $X$ and $Y$ are pointed spaces
then $[X,Y]_*$ will denote the set of pointed homotopy classes of such maps.
\subsection{$L_\infty$ algebras}
We will briefly recall the definition of an $L_\infty$ algebra. The reader is referred to \cite{Laz} for more detailed treatment, which includes the description of representing cdgas of $L_\infty$ algebra as Hinich cofibrant objects.

Let $V$ be a graded vector space. Let $\Der(\hat{S}\Sigma^{-1}V^*)$ be the graded Lie algebra consisting of continuous derivations of the completed symmetric algebra $\hat{S}\Sigma^{-1}V^*$. The derivations having vanishing constant term will be denoted by $\overline{\Der}(\hat{S}\Sigma^{-1}V^*)$; they form a sub-dgla in $\Der(\hat{S}\Sigma^{-1}V^*)$.
\begin{defi}\
\begin{itemize}
\item
An $L_\infty$ algebra supported on $V$ is an MC element $m_V\in \overline{\Der}(\hat{S}\Sigma^{-1}V^*)$. The element $m$ can be written as a sum $m_V=m=m_1+m_2+\ldots $ where $m_n=(m_V)_n$ is the part of $m$ of degree $n$ so we can write $m_n:\Sigma^{-1} V^*\to \hat{S}^n(\Sigma^{-1}V^*)$. The maps $m_n$ (or, rather, their duals) $(\Sigma V)^{\otimes n}\to \Sigma V$ are also called higher products in $V$. The pair $(V,m)$ will often be referred to as simply `$L_\infty$ algebra'.
\item
The derivation $m$ squares to zero and makes $\hat{S}\Sigma^{-1}V^*$ into a complete cdga; it will be called the \emph{representing} cdga of $V$. Note that $m_1$ is a differential on the graded vector space $V$; if it vanishes then the corresponding $L_\infty$ algebra is called \emph{minimal}.
\item
An $L_\infty$ map $f:V\to W$ is, by definition, a (continuous) map between the corresponding representing complete cdgas so that
$f:\hat{S}\Sigma^{-1}W^*\to \hat{S}\Sigma^{-1}V^*$. The degree $n$ part of $f$ will be denoted by $f_n$ so that $f_n:\Sigma^{-1}W^*\to \hat{S}^n(\Sigma^{-1}V^*)$; we will also use the same symbol for the dual map $(\Sigma V)^{\otimes n}\to \Sigma W$. An $L_\infty$ map $f$ is a \emph{weak equivalence}, or an \emph{$L_\infty$ quasi-isomorphism} if $f_1:\Sigma V\to \Sigma W$ is a quasi-isomorphism with respect to the differentials $(m_V)_1$ and $(m_W)_1$.
\end{itemize}
\end{defi}
The representing cdga of an $L_\infty$ algebra is just a cofibrant object in the Hinich closed model category of complete cdgas and thus, one is entitled to form the set $[V,W]_{L_\infty}$ of homotopy classes of $L_\infty$ maps $V\to W$. An explicit notion of (Sullivan) homotopy is defined in the usual way using the polynomial de Rham algebra of forms on an interval. We refer the reader to \cite{Laz} for details.

\subsection{Truncations of $L_\infty$ algebras}
Let $V$ be a homologically graded vactor space.
\begin{defi}
An $L_\infty$ algebra $(V,m)$ is called $n$-connected if $V_i=0$ for $i<n$.
\end{defi}
For any $L_\infty$ algebra $V$ and $n\geq 0$ there exists an $L_\infty$ algebra $V\langle n\rangle$ which is an $L_\infty$ analogue of the $n$-connected cover of a topological space. This construction was introduced in \cite{Laz} for $n=0$ but it goes through with obvious modifications for general $n\geq 0$. In our applications we will need the case $n=1$.
\begin{defi}
Let $(V,m_V)$ be an $L_\infty$ algebra; then its \emph{$n$-connected cover} is the $L_\infty$ algebra ${V}\langle n\rangle$ defined by the formula
\[{V}\langle n\rangle_i=\begin{cases}V_i, \text{~if~ } i>n\\ \Ker\{m_1:V_n\to V_{n-1}\} \text{~if~ } i=n\\0 \text{~if~} i<n\end{cases}.\]
The $L_\infty$ structure $m_{V\langle n\rangle}$ on $V\langle n\rangle$ is the obvious restriction of $m_V$.
\end{defi}
It is clear that there is a natural (strict) $L_\infty$ map ${V}\langle n\rangle\to V$, moreover for two $L_\infty$ quasi-isomorphic $L_\infty$ algebras $V$ and $U$ the corresponding $n$-connected covers ${V}\langle n\rangle$ and ${U}\langle n\rangle$ are $L_\infty$ quasi-isomorphic.
The following result is proved in the same way as Proposition 5.11 of \cite{Laz}.
\begin{prop}\label{conn}
For any $n$-connected $L_\infty$ algebra $V$ and an $L_\infty$ algebra $W$ there is a natural bijection $[V,W]_{L_\infty}\cong[V,W\langle n\rangle]_{L_\infty}$.
\end{prop}
\noproof
\begin{rem}
The notion of an $n$-connected cover of an $L_\infty$ algebra is a direct analogue of the corresponding topological notion and so, it is most naturally formulated in the homologically graded context. Later on, however, we will need to consider this notion for \emph{cohomologically graded $L_\infty$ algebras}. The difference is, of course, purely terminological.
\end{rem}
\section{Chevalley-Eilenberg and Harrison cohomology  }
In this section we give definitions of the Chevalley-Eilenberg cohomology of $L_\infty$ algebras and of the Harrison-Andr\'e-Quillen (to be abbreviated to Harrison) cohomology of complete cdgas in a way that makes manifest the Gerstenhaber bracket on the corresponding complexes. In the rest of the paper we will use shorthand `CE' for `Chevalley-Eilenberg'.
\begin{defi}Let $(V,m)$ be an $L_\infty$ algebra with representing complete cdga $\hat{S}\Sigma^{-1}V^*$.
\begin{itemize}
\item The CE complex of $V$ is the cohomologically graded
dg vector space \[C_{\CE}(V,V):=\Sigma^{-1}\Der(\hat{S}\Sigma^{-1}V^*).\]
\item
The \emph{truncated} CE complex of $V$ is the
cohomologically graded
dg vector space \[\overline{C}_{\CE}(V,V):=\Sigma^{-1}\overline{\Der}(\hat{S}\Sigma^{-1}V^*).\]
\end{itemize}
\end{defi}
We now give a parallel definition of the Harrison cohomology of complete cdgas. For a complete cdga $A$ there exists a dgla $\L(A)$ whose underlying space is the free Lie algebra on $\Sigma A^*$ and the differential is induced by the product and the differential on $A$.
As a preparation for the definition of the Harrison cohomology let $\g$ be any cohomologically graded dgla and $\g\langle\tau\rangle$ be the dgla whose underlying graded Lie algebra is $\g$ with the freely adjoined variable $\tau$ with $|\tau|=1$. The differential $d^\tau$ in $\g\langle\tau\rangle$ is defined as follows. For $g\in\g$ we set $d^\tau(g)=d(g)+[g,\tau]$ and $d^\tau(\tau)=\frac{1}{2}[\tau,\tau]$.  The passage from $\g$ to $\g\langle\tau\rangle$ is the Lie analogue of adjoining the unit to an associative algebra. Further, denote by $\Der_\tau(\g\langle\tau\rangle)\subset \Der(\g\langle\tau\rangle)$ the graded Lie subalgebra consisting of those derivations whose image is contained in $\g \subset \g\langle\tau\rangle$. It is clear that $\Der_\tau(\g\langle\tau\rangle)$ is a dgla with the commutator bracket; moreover we have an isomorphism of dglas
$\Der_\tau(\g\langle \tau\rangle)\cong \Der(\g)\ltimes \g$, the semidirect product of $\Der(\g)$ and $\g$.

The following result shows that for a dgla $\g$ the CE complexes $C_{\CE}(\g,\g)$ and $\overline{C}_{\CE}(\g,\g)$ are a kind of left derived functors. This is a standard result, however we are not aware of any published reference.
\begin{lem}\label{CEcoh} Let $\g$ be a cofibrant dgla.
\begin{enumerate}
\item
The dgla $\Sigma\overline{C}_{\CE}(\g,\g)$ is quasi-isomorphic to the dgla $\Der(\g)$.
\item
The dgla $\Sigma{C}_{\CE}(\g,\g)$ is quasi-isomorphic to the dgla $\Der_\tau(\g\langle \tau\rangle)$.
\end{enumerate}
\end{lem}
\begin{proof}
It suffices to consider the case when $\g$ is standard cofibrant, i.e. it is free on a certain collection of generations and has a filtration with respect to which the associated graded Lie algebra is a free Lie algebra with vanishing differential. To prove (1) let $\hat{S}\Sigma^{-1}\g^*$ be the representing complete cdga of $\g$ and recall that the CE complex $\overline{C}_{\CE}(\g,\g)$ is defined as $\Sigma \overline{C}_{\CE}(\g,\g)=\overline{\Der}(\hat{S}\Sigma^{-1}\g^*)$, the derivations with vanishing constant term. Considering derivations of $\g$ as \emph{linear} derivations of $\hat{S}\Sigma^{-1}\g^*$ we have an inclusion of dglas $\Der(\g)\hookrightarrow \Sigma\overline{C}_{\CE}(\g)$; we wish to prove that this map is a quasi-isomorphism.

Note that the given filtration on $\g$ gives a filtration on $\Der(\g)$ where we say that $\xi\in\Der(\g)$ has filtration degree $q$ if it raises the filtration degree of any element in $\g$ by at most $q$. We have a similar filtration on $\Sigma\overline{C}_{\CE}(\g)$; the inclusion $\Der(\g)\hookrightarrow \Sigma\overline{C}_{\CE}(\g)$ is a map of filtered dglas and to show that it is a quasi-isomorphism it suffices to show that it is so on the associated graded dglas. This, in turn, reduces to showing that the CE cohomology of free (graded) Lie algebras vanish in CE degrees $>1$, which is well-known.

Finally, part (2) follows from (1) since $\Sigma C_{\CE}(\g,\g)\cong \Sigma\overline{C}_{\CE}(\g,\g)\ltimes \g$.
\end{proof}
\begin{defi}Let $A$ be a complete non-unital cdga.
\begin{itemize}
\item The Harrison complex of $A$ is the cohomologically graded
dg vector space
\[C_{\Harr}(A,A):=\Sigma^{-1}\Der_\tau(\L({A})\langle\tau\rangle).\]
\item
The \emph{truncated} Harrison complex of $A$ is the cohomologically graded
dg vector space
\[\overline{C}_{\Harr}(A,A):=\Sigma^{-1}\Der(\L(A)).\]
\end{itemize}
\end{defi}
Note that for an $L_\infty$ algebra $V$ the dg space $\Sigma C_{\CE}(V,V)$ has the structure of a dgla with respect to the commutator bracket and $\Sigma \overline{C}_{\CE}(V,V)\subset \Sigma C_{\CE}(V,V)$ is an inclusion of sub dglas. Similarly for a complete cdga $A$ we have an inclusion of dglas $\Sigma \overline{C}_{\Harr}(A,A)\subset \Sigma C_{\Harr}(A,A)$. The following result summarizes properties of CE and Harrison cohomology. 
\begin{theorem}\label{CEH}\
\begin{enumerate}
\item Let $V$ and $U$ be two $L_\infty$ quasi-isomorphic $L_\infty$ algebras. Then the dglas $\Sigma C_{\CE}(V,V)$ and $\Sigma C_{\CE}(U,U)$ are $L_\infty$ quasi-isomorphic. Similarly the dglas $\Sigma \overline{C}_{\CE}(V,V)$ and $\Sigma \overline{C}_{\CE}(U,U)$ are $L_\infty$ quasi-isomorphic.
\item
Let $A$ and $B$ be two weakly equivalent complete cdgas. Then the dglas $\Sigma C_{\Harr}(A,A)$ and $\Sigma C_{\Harr}(B,B)$  are $L_\infty$ quasi-isomorphic. Similarly the dglas $\Sigma \overline{C}_{\Harr}(A,A)$ and $\Sigma \overline{C}_{\Harr}(B,B)$ are $L_\infty$ quasi-isomorphic.
\item
Let $A$ be a complete non-unital cdga and consider the corresponding cofibrant dgla $\L(A)$. The the following dglas are $L_\infty$ quasi-isomorphic:
\[
\Sigma C_{\CE}(\L(A),\L(A))\simeq \Sigma C_{\Harr}(A,A);
\]
\[
\Sigma \overline{C}_{\CE}(\L(A),\L(A))\simeq \Sigma\overline{C}_{\Harr}(A,A)
\]
\item
Let $V$ be an $L_\infty$ algebra and $\hat{S}\Sigma^{-1}V^*$ be its representing complete cdga. Then the following dglas are $L_\infty$ quasi-isomorphic:
\[
\Sigma C_{\CE}(V,V)\simeq \Sigma C_{\Harr}(\hat{S}\Sigma^{-1}V^*,\hat{S}\Sigma^{-1}V^*);
\]
\[
\Sigma \overline{C}_{\CE}(V,V)\simeq \Sigma\overline{C}_{\Harr}(\hat{S}\Sigma^{-1}V^*,\hat{S}\Sigma^{-1}V^*)
\]
\end{enumerate}
\end{theorem}
\begin{proof}
We will start with (1). Note that a given $L_\infty$ quasi-isomorphism $U\to V$ \emph{does not} induce any map between $ C_{\CE}(V,V)$ and $ C_{\CE}(U,U)$. The standard argument using the closed model category structure on complete cdgas (cf. for example \cite{BL}, Theorem 2.8) can be used
to tackle this difficulty. Let $f:\hat{S}\Sigma^{-1}V^*\to \hat{S}\Sigma^{-1}U^*$ be a weak equivalence of complete cdgas representing the $L_\infty$ algebras $U$ and $V$; note that these are both cofibrant complete cdgas. Next, representing $f$ as a composition of a cofibration and a fibration of complete cdgas we reduce the problem to the case when $f$ itself is either fibration or a cofibration. Let us suppose that $f$ is a fibration; the case of a cofibration is considered similarly. Then $f$ has a right splitting, i.e. a map $g$ such that $f\circ g=\id$. Using the map $g$ we define an injective map of dglas $\Der(\hat{S}\Sigma^{-1}U^*)\to \Der(\hat{S}\Sigma^{-1}V^*)$. That this map is a quasi-isomorphism follows from the fact that both $\Der(\hat{S}\Sigma^{-1}U^*)$ and $\Der(\hat{S}\Sigma^{-1}V^*)$ are quasi-isomorphic to $\Der(\hat{S}\Sigma^{-1}V^*,\hat{S}\Sigma^{-1}U^*)$ of derivations of $\hat{S}\Sigma^{-1}V^*$ in $\hat{S}\Sigma^{-1}U^*$ where the $\hat{S}\Sigma^{-1}U^*$-module structure on $\hat{S}\Sigma^{-1}U^*$ is given by the map $f$. Note that we have actually constructed a \emph{zig-zag} of quasi-isomorphisms between the dglas $\Der(\hat{S}\Sigma^{-1}U^*)$ and $ \Der(\hat{S}\Sigma^{-1}V^*)$. It follows that the dglas $\Sigma C_{\CE}(V,V)$ and $\Sigma C_{\CE}(U,U)$ are $L_\infty$ quasi-isomorphic. The argument for $ \overline{C}_{\CE}(V,V)$ and $ \overline{C}_{\CE}(U,U)$ is the same.

To prove (2) note that for a complete cdga $A$ the dgla $\L(A)$ is cofibrant in the closed model category of dglas. Moreover, a weak equivalence between two complete cdgas $A$ and $B$ gives rise (by definition) to a quasi-isomorphism between the dglas $\L(A)$ and $\L(B)$. Then the same formal argument as for part (1) can be used to establish part (2).

To prove (3) one only has to note that $\L(A)$ is a cofibrant dgla, use Lemma \ref{CEcoh} and the definitions of $C_{\Harr}(A,A)$ and $\overline{C}_{\Harr}(A,A)$.

Part (4) is a consequence of (3), (1) and (2). Let us, for example, prove the that the dglas $\Sigma \overline{C}_{\CE}(V,V)$ and $\Sigma\overline{C}_{\Harr}(\hat{S}\Sigma^{-1}V^*,\hat{S}\Sigma^{-1}V^*)$ are $L_\infty$ quasi-isomorphic. By (1) and (2) we may replace $V$ by a cofibrant dgla of the form $\C(A)$ for some complete cdga $A$. Then the desired $L_\infty$ quasi-isomorphism takes the form:
\begin{equation}\label{des}
\Sigma \overline{C}_{\CE}(\L(A),\L(A))\simeq \Sigma\overline{C}_{\Harr}(\C\L(V),\C\L(V)).
\end{equation}
But the complete cdgas $\C\L(A)$ and $A$ are weakly equivalent and using (2) again we reduce (\ref{des}) to the following $L_\infty$ quasi-isomorphism:
\[
\Sigma \overline{C}_{\CE}(\L(A),\L(A))\simeq \Sigma\overline{C}_{\Harr}(A,A),
\]
which is true by (3).
\end{proof}
\begin{rem}\label{fibs}\
\begin{itemize}\item
Let $V$ be an $L_\infty$ algebra. Then it was proved\footnote{The grading conventions in \cite{CL} differ by a shift from the present ones.} in \cite{CL} that the inclusion of dglas $\Sigma\overline{C}_{\CE}(V,V)\to \Sigma C_{\CE}(V,V)$ can be extended to a homotopy fiber sequence of $L_\infty$ algebras $V\to\Sigma\overline{C}_{\CE}(V,V)\to\Sigma C_{\CE}(V,V)$ (note that the map $V\to\Sigma\overline{C}_{\CE}(V,V)$ is a (non-strict) $L_\infty$ map). There is also an analogue of this statement for $A_\infty$ algebras and Hochschild complexes. The results of this section suggests a parallel treatment for the Harrison cohomology. Indeed, for a complete non-unital cdga $A$ we have a homotopy fiber sequence of $L_\infty$ algebras
$\L(A)\to\Sigma\overline{C}_{\Harr}(A,A)\to \Sigma{C}_{\Harr}(A,A)$. It is very likely that this fiber sequence continues to hold for non-complete cgdas $A$ and also for $C_\infty$ algebras but we will not pursue this issue further.
\item
It is likely that Theorem \ref{CEH} admits a suitable generalization in the context of homotopy algebras over a pair of Koszul dual operads. In this connection we mention the recent prerint \cite{DW}, Theorem 3.1 where part (1) of Theorem \ref{CEH} was generalized in this way.
\end{itemize}
\end{rem}
\section{Extensions of $L_\infty$ algebras and their classification}
In this section we will study in some detail the notion of extension of $L_\infty$ algebras introduced independently in \cite{Me} and \cite{CL'}. An essentially equivalent notion was considered by Merkulov, \cite{Mer} under the name \emph{open-closed homotopy Lie algebra}.
\begin{defi}\
\begin{enumerate}
\item
Let $(V,m)$ be an $L_\infty$ algebra; then a subspace $I\subset V$ is called an $L_\infty$ ideal in $V$ if for any $n$ we have $m_n(\Sigma v_1,\ldots,\Sigma v_n)\in\Sigma I$ as long as at least one of the $v_i$'s belongs to $I$.
In that case both $I$ and $U=V/I$ inherit structures of $L_\infty$ algebras.
\item The resulting sequence of $L_\infty$ algebras and maps
\[
I\to V\to U
\]
is called an \emph{extension} of $U$ by $I$ (or having $I$ as the kernel, or fiber).
\end{enumerate}
\end{defi}
\begin{rem}
It is clear that the classical notion of a Lie algebra extension is a special case of the above definition. It is, perhaps, more instructive to view the notion of an $L_\infty$ extension as a \emph{deformation with a dg base}. Indeed, let $\hat{S}\Sigma^{-1}I^*$, $\hat{S}\Sigma^{-1}V^*$ and $\hat{S}\Sigma^{-1}U^*$ be the representing cdgas of the $L_\infty$ algebras $I,V$ and $U$ respectively. Choosing a splitting of vector spaces $V\cong U\oplus I$ we write \[\hat{S}\Sigma^{-1}V^*\cong \hat{S}\Sigma^{-1}I^*\otimes \hat{S}\Sigma^{-1}U^*.\] Thus, the cdga $\hat{S}\Sigma^{-1}V^*$ becomes a $\hat{S}\Sigma^{-1}U^*$-linear cdga. Moreover, the condition that $I$ is an $L_\infty$ ideal in $V$ implies that the quotient of $\hat{S}\Sigma^{-1}V^*$ by the ideal generated by $\Sigma^{-1} U^*$ is isomorphic to $\hat{S}\Sigma^{-1}I^*$, the representing cdga of the $L_\infty$ algebra $I$. Thus, the $L_\infty$ algebra $V$ can be considered as a (generalized) deformation of the $L_\infty$ algebra $I$ over the dg base $(U,m_U)$. For example, if $U$ is the one-dimensional $L_\infty$ algebra sitting in degree $-1$ then we recover the usual notion of a one-parameter formal deformation of $V$, i.e. an $L_\infty$ algebra over the ring $\ground[[t]], |t|=0$  which is topologically free over $\ground[[t]]$ and reduces to $I$ upon setting $t=0$.
\end{rem}
We will have a chance to use the notion of a \emph{homotopy fiber sequence} of $L_\infty$ algebras which is, as we will see, essentially equivalent to that of an $L_\infty$ extension. Recall \cite{Laz} that representing cdgas of $L_\infty$ algebras are cofibrant objects in Hinich's closed model category of complete cgdas.
\begin{defi}
Let $\h\to\g\to\mathfrak f$ be a sequence of $L_\infty$ algebras. It is called a homotopy fiber sequence of $L_\infty$ algebras if the corresponding
sequence of representing cdgas \[\hat{S}\Sigma^{-1}{\mathfrak f}^*\to \hat{S}\Sigma^{-1}{\mathfrak g}^*\to\hat{S}\Sigma^{-1}{\mathfrak h}^*\] is equivalent, in the homotopy category of complete cdgas, to a sequence $A\to B\to C$ where $A\to B$ is a cofibration of complete cdgas and $C\cong {B}\otimes_A\ground$ is the corresponding cokernel.
\end{defi}
\begin{rem}
Let $\h\to\g\to\mathfrak f$ be a usual short exact sequence of dglas (e.g. ordinary Lie algebras). This is clearly an example of a homotopy fiber sequence of $L_\infty$ algebras. Moreover, using Hinich's equivalence of homotopy categories of dglas and complete cdgas any homotopy fiber sequence of $L_\infty$ algebras gives a short exact sequence of dglas and vice-versa.
\end{rem}
\begin{prop}\label{fibseq}
Let $e:I\to V\to U$ be an extension of $L_\infty$ algebras. Then it is a homotopy fiber sequence  of $L_\infty$ algebras. Conversely, any homotopy fiber sequence of $L_\infty$ algebras is homotopy equivalent to an $L_\infty$ extension.
\end{prop}
\begin{proof}
Let $e:I\to V\to U$ be an $L_\infty$ extension. Denote by $f:\hat{S}\Sigma^{-1}{U}^*\to \hat{S}\Sigma^{-1}{V}^*$ the map on the level of representing cdgas that corresponds to the map $V\to U$ in $e$.  We claim that $f$ is a cofibration  of complete cdgas; this will clearly imply that $e$ is a homotopy fiber sequence of $L_\infty$ algebras. To see that consider the following diagram of complete cdgas:
\[
\xymatrix{\C\L\hat{S}\Sigma^{-1}{U}^*\ar^{\C\L(f)}[r]\ar[d]& \C\L\hat{S}\Sigma^{-1}{V}^*\ar[d]\\
\hat{S}\Sigma^{-1}{U}^*\ar^f[r]& \hat{S}\Sigma^{-1}{V}^*}
\]
Here $\L$ and $\C$ are the adjoint pair of functors between the categories of complete cdgas and dglas (also known as the Harrison and CE complexes, cf. \cite{Laz} for a detailed discussion) which establish a Quillen equivalence between these closed model categories.
The downward vertical maps in the above diagram are fibrations and weak equivalences of complete cdgas and since $\hat{S}\Sigma^{-1}{U}^*$
and $\hat{S}\Sigma^{-1}{V}^*$ are cofibrant by \cite{Laz}, Proposition 3.3, these maps admit sections and it follows that the map
$f:\hat{S}\Sigma^{-1}{U}^*\to \hat{S}\Sigma^{-1}{V}^*$ is a retract of $\C\L(f):\C\L\hat{S}\Sigma^{-1}{U}^* \to\C\L\hat{S}\Sigma^{-1}{V}^*$. Next, the map $\L(f):\L\hat{S}\Sigma^{-1}{U}^*\to\L\hat{S}\Sigma^{-1}{V}^*$ is clearly surjective and thus, is a fibration of dglas. Since the functor $\C$ converts fibrations of dglas into cofibrations of complete cdgas we conclude that $\C\L(f)$ is a cofibration and thus, so is its retract $f$.

Conversely, let $\h\to\g\to\mathfrak f$ be a homotopy fiber sequence of $L_\infty$ algebras; we wish to show that it is homotopy equivalent to an $L_\infty$ extension. Denote the representing map of $\g\to \mathfrak f$ by $f:\hat{S}\Sigma^{-1}{\f}^*\to \hat{S}\Sigma^{-1}{\g}^*$. Without loss of generality we can assume the map $f$ is a cofibration of complete cdgas (and thus, injective) and $\hat{S}\Sigma^{-1}{\h}^*$ is the cokernel of $f$. We have the following commutative diagram of complete cdgas
\[
\xymatrix{\C\L\hat{S}\Sigma^{-1}\f^*\ar^{\C\L(f)}[r]\ar[d]& \C\L\hat{S}\Sigma^{-1}{\g}^*\ar[d]\\
\hat{S}\Sigma^{-1}{\f}^*\ar^f[r]& \hat{S}\Sigma^{-1}{\g}^*}
\]
where vertical maps are weak equivalences of complete cdgas. Note that the dgla map $\L(f):\L\hat{S}\Sigma^{-1}\f^*\to\L\hat{S}\Sigma^{-1}{\g}^*$ is surjective; denote its kernel by $I$. Then the sequence of dgla maps $I\to \L\hat{S}\Sigma^{-1}\f^*\to\L\hat{S}\Sigma^{-1}{\g}^*$ could be viewed as an $L_\infty$ extension which is clearly homotopy equivalent to the original sequence $\h\to\g\to\mathfrak f$.
\end{proof}
As usual, there is a concomitant notion of an \emph{equivalence} of two $L_\infty$ extensions. In fact, we have \emph{two} natural ways to
define equivalent extensions.
\begin{defi}\
Let $e:I\to V\to U$ and $e^\prime: I\to V^\prime\to U$ be two $L_\infty$ extensions.\begin{enumerate} \item The extensions $e$ and $e^\prime$ are called \emph{equivalent} if they could be included into a commutative diagram
\[
\xymatrix
{
e:\\
e^\prime:
}
\xymatrix
{
I\ar@{=}[d]\ar[r]&V\ar^f[d]\ar[r]&U\ar@{=}[d]\\
I\ar[r]&V^\prime\ar[r]&U.
}
\]
where $f$ is an $L_\infty$ isomorphism. Since $L_\infty$ isomorphisms are invertible and their composition is again an $L_\infty$ isomorphism, this is indeed an equivalence relation. The set of equivalence classes of such extensions will be denoted by $\Ext_{L_\infty}(U,I)$.
\item
The extensions $e$ and $e^\prime$ are called \emph{freely equivalent} if they could be included into a commutative diagram
\[
\xymatrix
{
e:\\
e^\prime:
}
\xymatrix
{
I\ar^g[d]\ar[r]&V\ar^f[d]\ar[r]&U\ar@{=}[d]\\
I\ar[r]&V^\prime\ar[r]&U.
}
\]
where $f$ and $g$ are $L_\infty$ isomorphisms. As in part (1), this defines an equivalence relation.The set of equivalence classes of such extensions will be denoted by $\widetilde{\Ext}_{L_\infty}(U,I)$.
\end{enumerate}
\end{defi}

\begin{rem}
The definition of equivalence becomes transparent when one regards $V$ and $V^\prime$ as deformations with base $U$. The extensions $e$ and $e^\prime$
are equivalent if there is an isomorphism $f$ of $\hat{S}\Sigma^{-1}U^*$-linear $L_\infty$ algebras given by
\[\xymatrix{
\hat{S}\Sigma^{-1}(V^\prime)^*=\hat{S}\Sigma^{-1}(I^\prime)^*\otimes \hat{S}\Sigma^{-1}(U^\prime)^*\ar^-f[r]&
\hat{S}\Sigma^{-1}I^*\otimes \hat{S}\Sigma^{-1}U^*=\hat{S}\Sigma^{-1}V^*
}
\]
which becomes the identity map when reduced modulo the ideal $(\Sigma^{-1}U^*)$. This is, therefore, the usual notion of equivalence of formal deformations.

The definition of a free equivalence is less familiar. It could be interpreted as an isomorphism between two deformations \emph{which is not necessarily identical on the central fiber}. Therefore, two equivalent extensions must be freely equivalent, but not vice-versa.
\end{rem}
The following result shows that the functor $U\mapsto \Ext_{L_\infty}(U,I)$ is representable in the homotopy category of $L_\infty$ algebras.
\begin{theorem}\label{class1}Let $U$ and $I$ be $L_\infty$ algebras.
Then there is a natural bijection $\Ext_{L_\infty}(U,I)\cong [U,\Sigma C_{\CE}(I,I)]_{L_\infty}$ between equivalence classes of
$L_\infty$ extensions of $U$ by $I$ and homotopy classes of $L_\infty$ maps from $U$ into the dgla
$\Sigma C_{\CE}(I,I)$ supplied with the Gerstenhaber bracket.
\end{theorem}
\begin{proof}
Let $\hat{S}\Sigma^{-1}U^*$ and $\hat{S}\Sigma^{-1}I^*$ be the representing cdgas of the $L_\infty$ algebras $U$ and $I$ respectively.
The data of an $L_\infty$ extension $e:I\to V\to U$ is tantamount to specifying an $\hat{S}\Sigma^{-1}U^*$-linear differential in
$\hat{S}\Sigma^{-1}U^*\otimes \hat{S}\Sigma^{-1}I^*$ which is the same as an MC element $\xi$ in the pronilpotent dgla $\hat{S}\Sigma^{-1}U^*_+\otimes\Sigma\CE(I,I)$.
Such an MC element is, in turn, equivalent to having an $L_\infty$ map $U\to \Sigma C_{\CE}(I,I)$. Two such extensions $e$ and $e^\prime$ are equivalent,
by definition, if and only if the corresponding MC elements $\xi$ and $\xi^\prime$ are gauge equivalent which, by the Schlessinger-Stasheff
theorem (cf. \cite{SS, CL'}), is equivalent to $\xi$ and $\xi^\prime$ being Sullivan homotopic. The theorem is proved.
\end{proof}
\begin{defi}
Let $I, U, W$ be $L_\infty$ algebras, $e:I\to V\to U$ be an $L_\infty$ extension with fiber $I$ corresponding to an $L_\infty$ map $f:U\to \Sigma C_{\CE}(I,I)$ and $g:W\to U$ be an $L_\infty$ map. Then the $L_\infty$ extension $I\to ?\to W$ corresponding to the composite map $f\circ g:W\to \Sigma C_{\CE}(I,I)$ will be called the $L_\infty$ extension induced from $e$ by $g$ and it will be denoted by $g^*(e)$. The extension $e_I$ corresponding to the identity map from $\Sigma C_{\CE}(I,I)$ into itself will be called the \emph{universal} $L_\infty$ extension with fiber $I$. Any other $L_\infty$ extension with fiber $I$ is thus induced from $e_I$.
\end{defi}
\begin{rem}
The proof of Theorem \ref{class1} makes clear that the correspondence between $L_\infty$ maps into $U\to \Sigma C_{\CE}(I,I)$ and $L_\infty$ extensions of the form $I\to ?\to U$ is one-to-one on the nose rather then up to an equivalence; moreover replacing this map with a homotopic one results in having an equivalent extension.
\end{rem}
We now prove a similar result classifying $L_\infty$ extensions up to free equivalence. To this end denote, for an $L_\infty$ algebra $I$, the group of its $L_\infty$ automorphisms by $\Aut_{L_\infty}(I)$. This group is, therefore, the group of automorphisms of the complete cdga $(\hat{S}\Sigma^{-1}I^*, m_I)$ representing $I$. It is clear that $\Aut_{L_\infty}(I)$ acts on $\Sigma C_{\CE}(I,I)\cong \Der(\hat{S}\Sigma^{-1}I^*)$ by conjugations. The classification of $L_\infty$ extensions up to free equivalence could now be formulated as follows.
\begin{theorem}\label{class2}Let $U$ and $I$ be $L_\infty$ algebras. Then there is a natural bijection between the set $\widetilde{\Ext}_{L_\infty}(U,I)$ and
the quotient $[U,\Sigma C_{\CE}(I,I)]_{L_\infty}\big/\big({\Aut_{L_\infty}(I)}\big)$ where the group $\Aut_{L_\infty}(I)$ acts on $[U,\Sigma C_{\CE}(I,I)]_{L_\infty}$ through its action on $\Sigma C_{\CE}(I,I)$.
\end{theorem}
\begin{proof}
Again, we interpret $L_\infty$ extensions $I\to V\to U$ as MC elements in the dgla $\hat{S}\Sigma^{-1}U^*_+\otimes \Sigma C_{\CE}(I,I)$. The difference is that the gauge group is bigger in this case. Namely, denote by $G\subset\Aut(\hat{S}\Sigma^{-1}U^*\otimes \hat{S}\Sigma^{-1}I^*)\cong \Aut(\hat{S}\Sigma^{-1}(U^*\oplus I^*))$ the subgroup of those automorphisms which fix $U^*$ and which are the identity modulo the ideal generated by $U^*$. Then $G$ is the pronilpotent group associated with the pronilpotent Lie algebra $[\hat{S}\Sigma^{-1}U^*_+\otimes \Sigma C_{\CE}(I,I)]_0$ and the action of this group on $\hat{S}\Sigma^{-1}U^*_+\otimes \Sigma C_{\CE}(I,I)$ identifies equivalent $L_\infty$ extensions. The free equivalence of $L_\infty$ extensions corresponds to the action of the group $\tilde{G}\subset \Aut(\hat{S}\Sigma^{-1}(U^*\oplus I^*))$ consisting of those automorphisms $g$ which still fix $U^*$ and such that the composite map
\[
\xymatrix{
\hat{S}\Sigma^{-1}I^*\ar@{^{(}->}[r]&\hat{S}\Sigma^{-1}(U^*\oplus I^*)\ar^g[r]&\hat{S}\Sigma^{-1}(U^*\oplus I^*)\ar^-{\mod U^*}[r]&
\hat{S}\Sigma^{-1}(I^*)
}
\]
is an automorphism of $\hat{S}\Sigma^{-1}(I^*)$. If this automorphism is the identity then the corresponding automorphism belongs to the group $G$. Therefore, we have a (split) extension of groups $\Aut_{L_\infty}(I)\to \tilde{G}\to G$. We have:
\begin{align*}
\widetilde{\Ext}_{L_\infty}(U,I)&\cong\MC(\hat{S}\Sigma^{-1}U^*_+\otimes \Sigma C_{\CE}(I,I))\big/\tilde{G}\\
&\cong [\MC(\hat{S}\Sigma^{-1}U^*_+\otimes \Sigma C_{\CE}(I,I))/{G}]\big/\big(\Aut_{L_\infty}(I)\big)\\
&\cong[\Ext_{L_\infty}(U,I)]\big/\big(\Aut_{L_\infty}(I)\big)\\
&\cong[U,\Sigma C_{\CE}(I,I)]_{L_\infty}\big/\big(\Aut_{L_\infty}(I)\big)
\end{align*}
as required.
\end{proof}
\begin{rem}Since the group $\Aut_{L_\infty}(I)$ is not nilpotent in general (e.g. when the $L_\infty$ structure on $I$ is the trivial one), we cannot use the Schlessinger-Stasheff theorem to interpret $\widetilde{\Ext}_{\L_\infty}(U,I)$ in terms of homotopy classes of MC elements. However, we will see later that in the context of rational homotopy theory there is, in fact, such an interpretation via homotopies of maps of unpointed spaces.
\end{rem}
\subsection{Comparison with the classical notion} We will now give a comparison of our classification result with the standard interpretation of Lie algebra extensions in terms of CE cohomology. Suppose that $U$ is an (ungraded) Lie algebra and $I$ is a $U$-module viewed as an abelian Lie algebra. An extension of $U$ by $I$ is a short exact sequence of Lie algebras of the form $e:I\to V\to U$ such that the action of the quotient Lie algebra $V/I\cong U$ on $I$ is the given one.
Two such extensions $e$ and $e^\prime$ are equivalent if there is a map $e\to e^\prime$ which is the identity map on $I$ and $U$.

To see how this is a special case of an $L_\infty$ extension consider the graded Lie algebra $\Sigma C_{\CE}(I,I)=\Der(\hat{S}\Sigma^{-1}I^*)$ and its subalgebra
$\g$ consisting of constant and linear derivations. Thus, as a graded vector space, $\g$ is isomorphic to $\Sigma I\oplus\Hom(I,I)$. The Lie bracket is zero on $\Sigma I$; on $\Hom(I,I)$ it is the usual commutator of endomorphisms, and for $g\in\Hom(I,I),v\in\Sigma I$ we have $[g,v]=g(v)$. Thus, $\g$ is just the Lie algebra of affine transformations of $\Sigma I$. Then an extension $e:I\to V\to U$ can be viewed as an $L_\infty$ extension whose classifying $L_\infty$ map $f:U\to \Sigma C_{\CE}(I,I)$ lands inside the subalgebra $\g$. Such a map $f$ has two components: $f_1:U\to\g$ and $f_2:U\otimes U\to\g$.
Dimensional considerations show that the images of $f_1$ and $f_2$ lie inside $\Hom(I,I)$ and $I$ respectively. Moreover, $f_1$ is a Lie algebra map and it specifies the action of $U$ on $I$ whereas $f_2$ is the cocycle classifying the extension $e$. Then, the homotopy classes of $L_\infty$ maps $U\to \g$ for which $f_1$ (the action) is fixed are precisely 2-dimensional cohomology of $U$ with coefficients in $I$, in agreement with the standard classification.

Note that to obtain such a simple homological interpretation of classical Lie algebra extensions we need the kernel $I$ to be abelian; otherwise the Lie algebra $\Sigma C_{\CE}(I,I)$ supports a non-zero (CE) differential and its subalgebra $\g$ is \emph{not closed} with respect to it. Therefore, it makes no sense to consider $L_\infty$ maps $U\to \g$. We see that the problem of classifying Lie algebra extensions with a \emph{non-abelian} kernel in homological terms forces one to extend the notion of extension in the $L_\infty$ direction.

\section{Universal extensions: examples}
Given an $L_\infty$ algebra $I$, we saw that there is a universal $L_\infty$ extension from which all other extensions having $I$ as the fiber are induced. We now examine more closely this and other, closely related extensions. Recall from Remark \ref{fibs} that there is a homotopy fiber sequence of $L_\infty$ algebras:
\begin{equation}\label{fibCE}
I\to \Sigma\overline{C}_{\CE}(I,I)\to \Sigma C_{\CE}(I,I).
\end{equation}
 It turns out that the sequence (\ref{fibCE}) is, essentially, the universal extension  with fiber $I$.
\begin{theorem}\label{univext}
Let $e:I\to V\to \Sigma C_{\CE}(I,I)$ be the universal $L_\infty$ extension with fiber $I$ corresponding to the identity map on $\Sigma C_{\CE}(I,I)$. Then $e$ is homotopy equivalent to the sequence (\ref{fibCE}).
\end{theorem}
\begin{proof}
The underlying vector space for the $L_\infty$ algebra $V$ is
\begin{align*}
V&\cong I\oplus \Sigma C_{\CE}(I,I)\\
&\cong I\oplus\Hom(\Sigma^{-1}I^*,\ground)\oplus\Hom(\Sigma^{-1}I^*,\Sigma^{-1}I^*)\oplus\Hom(\Sigma^{-1}I^*,S^2\Sigma^{-1}I^*)\oplus\ldots\\
&\cong I\oplus\Sigma I\oplus\Hom(\Sigma I,\Sigma I)\oplus \Hom(S^2\Sigma I,\Sigma I)\oplus\ldots.
\end{align*} Inspection shows that the operation $m_1:V\to V$ is only nonzero when applied to the subspace $\Sigma I$, and it maps it isomorphically onto $I$. Further, the subspace
\[\Hom(\Sigma I,\Sigma I)\oplus \Hom(S^2\Sigma I,\Sigma I)\oplus\ldots\subset V\]
is isomorphic to $\overline{C}_{\CE}(I,I)$ and is an $L_\infty$ subalgebra in $V$. It follows that the (strict) $L_\infty$ inclusion
$\overline{C}_{\CE}(I,I)\hookrightarrow V$ is a quasi-isomorphism and thus, $V$ is $L_\infty$ quasi-isomorphic to $\overline{C}_{\CE}(I,I)$. The desired conclusion follows.
\end{proof}
\begin{rem}
Theorem \ref{univext} is an algebraic analogue of the well-known result of Gottlieb \cite{Got} stating that the total space $E$ of the universal fibration $F\to E\to\Baut(F)$ classifying fibrations with a given fiber $F$ is itself (homotopy equivalent to) $\Baut_*(F)$, the classifying space of basepoint-preserving self-homotopy equivalences of $F$. This is, of course, more than an analogy: we will see that rational homotopy theory transforms the statement of Theorem \ref{univext} into Gottlieb's result.
\end{rem}
In light with the analogy with Gottlieb's result it is natural to ask whether, for a given $L_\infty$ algebra $I$, the dgla $\Sigma\overline{C}_{\CE}(I,I)$ also classifies a certain extension problem (note that the space $\Baut_*(F)$ classifies fibrations with fiber $F$ and a given section). It turns out that $\Sigma\overline{C}_{\CE}(I,I)$ classifies \emph{split} $L_\infty$ extensions, that is to say, $L_\infty$ extensions having a \emph{section}.
\begin{defi}
Let $e:I\to V\to U$ be an $L_\infty$ extension. We say that the extension $e$ is \emph{split} if there is a splitting $V\cong I\oplus U$ making $U$ an $L_\infty$ subalgebra of $U$. To emphasize the section we will write the diagram of $e$ as $\xymatrix{I\ar[r]& V\ar[r]&U\ar@/_/[l]}$
\end{defi}
\begin{rem}
Repeating the arguments of the proof of Proposition \ref{fibseq} (with obvious modifications) we see that a split $L_\infty$ extension can be viewed as a homotopy fiber sequence of $L_\infty$ algebras having a section and conversely, any homotopy fiber sequence of $L_\infty$ algebras with a section gives rise to a split $L_\infty$ extension.
\end{rem}
There are obvious notions of equivalence of split extensions.
\begin{defi}\
Let $e:\xymatrix{I\ar[r]& V\ar[r]&U\ar@/_/[l]}$ and $e^\prime: \xymatrix{I\ar[r]& V^\prime\ar[r]&U\ar@/_/[l]}$ be two split $L_\infty$ extensions.\begin{enumerate} \item The split extensions $e$ and $e^\prime$ are called \emph{equivalent} if they could be included into a commutative diagram
\[
\xymatrix
{
e:\\
e^\prime:
}
\xymatrix
{
I\ar@{=}[d]\ar[r]&V\ar^f[d]\ar[r]&U\ar@{=}[d]\ar@/_/[l]\\
I\ar[r]&V^\prime\ar[r]&U\ar@/_/[l].
}
\]
where $f$ is an $L_\infty$ isomorphism. The set of equivalence classes of such extensions will be denoted by $\Ext^s_{L_\infty}(U,I)$.
\item
The extensions $e$ and $e^\prime$ are called \emph{freely equivalent} if they could be included into a commutative diagram
\[
\xymatrix
{
e:\\
e^\prime:
}
\xymatrix
{
I\ar^g[d]\ar[r]&V\ar^f[d]\ar[r]&U\ar@{=}[d]\ar@/_/[l]\\
I\ar[r]&V^\prime\ar[r]&U\ar@/_/[l].
}
\]
where $f$ and $g$ are $L_\infty$ isomorphisms. The set of equivalence classes of such extensions will be denoted by $\widetilde{\Ext}^s_{L_\infty}(U,I)$.
\end{enumerate}
\end{defi}
We have the following analogue of Theorems \ref{class1}; the proof is similar, but takes into account the splitting datum.
\begin{theorem}\label{classs1}Let $U$ and $I$ be $L_\infty$ algebras.
Then there is a natural bijection $\Ext^s_{L_\infty}(U,I)\cong [U,\Sigma\overline{C}_{\CE}(I,I)]_{L_\infty}$ between equivalence classes of
split $L_\infty$ extensions of $U$ by $I$ and homotopy classes of $L_\infty$ maps from $U$ into the dgla
$\Sigma\overline{C}_{\CE}(I,I)$ supplied with the Gerstenhaber bracket.
\end{theorem}
\begin{proof}
Let $\hat{S}\Sigma^{-1}U^*$ and $\hat{S}\Sigma^{-1}I^*$ be the representing cdgas of the $L_\infty$ algebras $U$ and $I$ respectively.
The data of a split $L_\infty$ extension $e:I\to V\to U$ is equivalent to specifying an $\hat{S}\Sigma^{-1}U^*$-linear differential $m_V$ in
$\hat{S}\Sigma^{-1}U^*\otimes \hat{S}\Sigma^{-1}I^*$ such that the inclusion of dg spaces $U\to V$ is
an $L_\infty$ map. That means that the quotient map
\[
\hat{S}\Sigma^{-1}V^*\cong\hat{S}\Sigma^{-1}U^*\otimes \hat{S}\Sigma^{-1}I^*\to (\hat{S}\Sigma^{-1}U^*\otimes \hat{S}\Sigma^{-1}I^*)/(I^*)\cong\hat{S}\Sigma^{-1}U^*
\]
is a cdga map. This happens if and only if the image of $I$ under $m_V$ lies in the ideal generated by $I$. Therefore the associated
MC element $\xi$ in the pronilpotent dgla $\hat{S}\Sigma^{-1}U^*_+\otimes \Sigma C_{\CE}(I,I)$ actually lies in $\hat{S}\Sigma^{-1}U^*_+\otimes\Sigma\overline{\CE}(I,I)$. Such an MC element is, in turn, equivalent to having an $L_\infty$
map $U\to \Sigma\overline{C}_{\CE}(I,I)$.

Two such split extensions $e$ and $e^\prime$ are equivalent (as split extensions) if and only if the corresponding MC elements $\xi,\xi^\prime\in\MC(\hat{S}\Sigma^{-1}U^*_+\otimes\Sigma\overline{C}_{\CE}(I,I))$ are gauge equivalent which, by the Schlessinger-Stasheff
theorem, is equivalent to $\xi$ and $\xi^\prime$ being Sullivan homotopic. The theorem is proved.

\end{proof}
Note that the the action of $\Aut_{L_\infty}(I)$ on $C_{\CE}(I,I)$ restricts to an action on $\overline{C}_{\CE}(I,I)$ and there is an analogue
of Theorem \ref{class2}. The only modification of the proof is replacing $C_{\CE}(I,I)$ with $\overline{C}_{\CE}(I,I)$.
\begin{theorem}\label{classs2}Let $U$ and $I$ be $L_\infty$ algebras. Then there is a natural bijection between the set $\widetilde{\Ext}^s_{L_\infty}(U,I)$ and
the quotient $[U,\Sigma \overline{C}_{\CE}(I,I)]_{L_\infty}\big/\big({\Aut_{L_\infty}(I)}\big)$ where the group $\Aut_{L_\infty}(I)$ acts on $[U,\Sigma C_{\CE}(I,I)]_{L_\infty}$ through its action on $\Sigma C_{\CE}(I,I)$.
\end{theorem}
\noproof
\begin{rem}\label{higherCE}The passage from $C_{\CE}(I,I)$ to $\overline{C}_{\CE}(I,I)$ can be iterated. To this end let us redenote $C_{\CE}(I,I)$ by $C_{\CE}[0](I,I)$ and $\overline{C}_{\CE}(I,I)$ by $C_{\CE}[1](I,I)$. Further, denote the total space of the universal extension over $\Sigma C_{\CE}[1](I,I)=\Sigma\overline{C}_{\CE}(I,I)$ by $\Sigma C_{\CE}[2](I,I)$ so that we have an extension $e_1:\xymatrix{I\ar[r]&\Sigma\CE[2](I,I)\ar^{f_1}[r]&\Sigma C_{\CE}[1](I,I)}$. By induction we construct the extensions $e_3,e_4,\ldots$ so that
\[e_n=f_1^*(e_{n-1}):\xymatrix{I\ar[r]&\Sigma C_{\CE}[n+1](I,I)\ar^{f_n}[r]&\Sigma C_{\CE}[n](I,I)}.\]
The $L_\infty$ algebra $\Sigma C_{\CE}[2](I,I)$ was introduced (in the case when the $L_\infty$ structure on $I$ is trivial) in \cite{CL'}; it is characterized by the property that, roughly speaking, the MC elements in $\Sigma\CE[2](I,I)$ are $L_\infty$ structures on the dg vector space $I$ together with an MC element in it. It is not hard to prove that $\Sigma C_{\CE}[n](I,I)$ classifies $L_\infty$ extensions with fiber $I$ having $n$ sections up to an appropriate equivalence relation.

The topological analogue of $\Sigma C_{\CE}[n](I,I)$ is the classifying space of fibrations with typical fiber $F$ and having $n$ marked points; such a construction is reminiscent of moduli spaces of Riemann surfaces with marked points.
\end{rem}
\subsection{The point of view of deformation theory} As was already mentioned, the notion of an $L_\infty$ extension is essentially equivalent to that
of a deformation with a dg base. Let us now make this precise. Let $A$ be a complete unital cdga, $I$ be a vector space and $m\in\MC\big(\overline{\Der}(\hat{S}\Sigma^{-1}I^*)\big)$ be an $L_\infty$ algebra structure on $I$. Consider the
space \[{\Der}^A(A\otimes\hat{S}\Sigma^{-1}I^*)\subset \Der(A\otimes\hat{S}\Sigma^{-1}I^*)\] consisting of derivations $\xi$ such that:
\begin{enumerate} \item $\xi|_{A\otimes 1}=0$ \item $\xi|_{1\otimes\hat{S}\Sigma^{-1}I^*}=0\mod(A_+)$
\end{enumerate}
 Then ${\Der}^A(A\otimes\hat{S}\Sigma^{-1}I^*)$ is a dgla with the differential induced by that
on $I$ and the commutator bracket. Further, reduction modulo the ideal $A_+\subset A$ gives rise to a dgla map \[p:{\Der}^A(A\otimes\hat{S}\Sigma^{-1}I^*)\to\overline{\Der}(\hat{S}\Sigma^{-1}I^*).\]
Next, let \[G^A\subset\Aut(A\otimes\hat{S}\Sigma^{-1}I^*)\] be the subgroup of automorphisms $g$ of the complete graded algebra $A\otimes\hat{S}\Sigma^{-1}I^*$ such that
\begin{enumerate}\item $g|_{A\otimes 1}=\id$ \item $g|_{1\otimes\hat{S}\Sigma^{-1}I^*}=\id\mod(A_+)$.\end{enumerate}
\begin{defi}
A deformation of the $L_\infty$ algebra $(I,m)$ over $A$ is an element $m^A\in\MC\big(\Der^A(A\otimes\hat{S}\Sigma^{-1}I^*)\big)$ such that $p_*(m^A)=m$. The group $G^A$ acts on the set of such deformations and two deformations belonging to the same $G^A$-orbit are called \emph{equivalent}. The functor associating to a complete cdga $A$ the set of equivalence classes of deformation of $I$ will be denoted by $\defo_I$.
\end{defi}
It is clear that in the case when $A$ is a representing cdga of an $L_\infty$ algebra: $A=\hat{S}\Sigma^{-1}U^*$ a deformation of $I$ with base $A$ is precisely the same as an $L_\infty$ extension $I\to V\to U$ possessing an $L_\infty$ section $U\to V$ where there is a decomposition of vector spaces $V=I\oplus U$. Note that in the absence of such a section the $A$-linear derivation of $\hat{S}\Sigma^{-1}V^*$ determining an $L_\infty$ structure should be viewed as a deformation of $I$ as a \emph{curved} $L_\infty$ algebra.  Moreover, the notion of equivalence of deformations corresponds to the equivalence of extensions. In other words, we have the following natural bijection:
\[
\defo_I\cong\Ext^s_{L_\infty}(-,I).
\]
Next, the set of equivalence classes of deformations of $I$ over $A$ is clearly in one-to-one correspondence with $\MCmod(\Sigma\overline{C}_{\CE}(I,I),A)$, the MC moduli set of the dgla $\Sigma\overline{C}_{\CE}(I,I)$ with values in $A$. Such a set is bijective with the set of homotopy classes of maps of complete
cdgas $[C_{\CE}\big(\Sigma\overline{C}_{\CE}(I,I)\big),A]$ where $C_{\CE}\big(\Sigma\overline{C}_{\CE}(I,I)\big):=\hat{S}\overline{C}_{\CE}(I,I)$ is the representing complete cdga of the dgla $\Sigma\overline{C}_{\CE}(I,I)$, cf. \cite{Laz}. We obtain the following result which specializes to Theorem \ref{classs1} when $A$ is a representing cdga of an $L_\infty$ algebra $U$:
\begin{theorem}\label{defo}
The functor $A\mapsto \defo_I(A)$ associating to a complete cdga $A$ the set of equivalence classes of deformations of an $L_\infty$ algebra $I$ over $A$ is represented by the complete cdga $C_{\CE}\big(\Sigma\overline{C}_{\CE}(I,I)\big)$, the CE complex of the Gerstenhaber dgla $\Sigma\overline{C}_{\CE}(I,I)$.
\end{theorem}
\noproof
\begin{rem}\
\begin{enumerate}
\item
Later on we will employ the notation $\Def_V$ for the deformation functor governed by an $L_\infty$ algebra $V$. With this convention we will have $\defo_I\cong\Def_{\overline{C}(I,I)}$.
\item
We see, that given an $L_\infty$ algebra $I$ there exists its universal deformation with base $C_{\CE}\big(\Sigma\overline{C}_{\CE}(I,I)\big):=\hat{S}\overline{C}_{\CE}(I,I)$. This universal deformation can be identified with the $L_\infty$ algebra $\Sigma C_{\CE}[2](I,I)$ introduced in Remark \ref{higherCE} whose underlying space is $I\oplus\Sigma\overline{C}_{\CE}(I,I)$.
\end{enumerate}
\end{rem}

\section{Rational models for classifying spaces}
In this section we apply the theory of $L_\infty$ extensions to constructing rational homotopy models for classifying spaces of fibrations. We will use standard results from rational homotopy theory, particularly the equivalence between homotopy categories of rational nilpotent spaces, of cdgas and dglas (more precisely, between certain subcategories inside them), cf. \cite{Q',BG, N}.  Let $F$ be
a connected nilpotent CW complex which is rational and rationally of finite type, $\Aut(F)$ be the monoid of homotopy self-equivalences of $F$ and $\Baut F$ be the classifying space of $\Aut(F)$. It will be convenient for us to choose a model for $\Aut(F)$ that is a topological group (which could be done since $\Aut(X)$ is a group-like monoid); of course the final formulations will be homotopy invariant and so will not depend on the choice of a model.  Recall that $\Baut(F)$ is the base of the universal fibration with fiber $F$. The fundamental group of $\Baut(F)$ is the group $\Haut(F)$ of self-homotopy equivalences of $F$ up to homotopy and it is typically non-nilpotent. To obtain a satisfactory algebraic model one should, therefore, pass to the simply-connected cover $\Baut(F)\langle 1\rangle$ (which is weakly equivalent to the classifying space of the submonoid of $\Aut(B)$ consisting of self-homotopy equivalences homotopic to the identity). The space $\Baut(F)$ is then recovered as the homotopy quotient of $\Baut(F)\langle 1\rangle$ by the group $\Haut(F)$.

We refer to the monograph \cite{Rud} for the the classification theory of fibrations, including its variants for sectioned fibrations and rooted
fibrations.  Recall from op.cit. that a rooted $F$-fibration is a (Hurewicz)  fibration $p:X\to Y$ over a pointed space $Y=(Y,y_0)$ together with a fixed homotopy equivalence $i_p:p^{-1}(y_0)\simeq F$. Two rooted $F$-fibrations $p:X\to Y,q:Z\to Y$ over the same pointed space $Y=(Y,y_0)$ are called \emph{equivalent} if there exists a homotopy equivalence  $f:X\to Z$ over $Y$ such that $i_{q}\circ f=i_p$. Since any map can be replaced by an equivalent fibration the notion of a rooted equivalence could be defined on the level of homotopy categories, that is to say that two homotopy fiber sequences $F\to X\to Y$ and $F\to Z\to Y$ are equivalent if there exists the following homotopy commutative diagram.
\[
\xymatrix{F\ar@{=}[d]\ar[r]&X\ar_{\simeq}[d]\ar[r]&Y\ar@{=}[d]\\
F\ar[r]&Z\ar[r]&F}
\]
We will denote by $\Fib(Y,F)$ the set of equivalence classes of rooted fibration over $Y$ with fiber $F$ (or, equivalently, the set of equivalence classes of homotopy fiber sequences of the form $F\to X\to Y$ as above. The functor $Y\to\Fib(Y,F)$ is known to be represented in the homotopy category of pointed spaces by $\Baut(F)$, \cite{Rud}.

We denote the Sullivan model of $F$ by $A(F)$ and the Lie-Quillen model of $F$ by $L(F):=\L(A(X)_+)$.
The model of $\Baut(F)$ will be obtained in terms of the CE cohomology of $L(F)$ or, equivalently, in terms of Harrison cohomology of $A(F)$.

Recall that earlier in the paper we have given the CE cohomology the traditional cohomological grading, and we retain this convention despite the present interpretation
as a Lie-Quillen model of a classifying space (which naturally has \emph{homological} grading).  The same remark applies to the Harrison complexes.

Since the complexes $C_{\CE}(L(F),L(F))$ and $C_{\Harr}(A(F),A(F))$ contain elements of arbitrary integer grading, we need to truncate them. The connected covers of dglas $\Sigma C_{\CE}(L(F),L(F))$ and $\Sigma C_{\Harr}(A(F),A(F))$ do not have elements in negative degrees but, since their degree zero parts usually form non-nilpotent Lie algebras, we need to go further and take the simply-connected covers
$\Sigma C_{\CE}(L(F),L(F))\langle 1\rangle$ and $\Sigma C_{\Harr}(A(F),A(F))\langle 1\rangle$. Then the following result holds.
\begin{theorem}\label{mainth} Let $F$ be a rational nilpotent CW complex having rationally finite type. Let $\Aut(F)$ and $\Aut_*(F)$ be the monoid of self-homotopy equivalences of $F$ and the monoid of basepoint-preserving self-homotopy equivalences of $F$ respectively. Let $A(F)$ and $L(F)$ be the Sullivan and Lie-Quillen models of $F$ respectively.
\begin{enumerate}
\item
The dglas $\Sigma C_{\CE}(L(F),L(F))\langle 1\rangle$ and $\Sigma C_{\Harr}(A(F),A(F))\langle 1\rangle$ are both Lie-Quillen models for $\Baut(F)\langle 1\rangle$.
\item
The dglas $\Sigma \overline{C}_{\CE}(L(F),L(F))\langle 1\rangle$ and $\Sigma \overline{C}_{\Harr}(A(F),A(F))\langle 1\rangle$ are both Lie-Quillen models for $\Baut_*(F)\langle 1\rangle$.
\item
Suppose that, in addition, $F$ either has a finite Postnikov tower or it is rationally equivalent to a finite CW complex.
\begin{enumerate}
\item
The group $\pi_1\Baut(F)\cong \Haut(F)$ is a group of $\mathbb Q$-points of an affine algebraic group over $\mathbb Q$ whose Lie algebra
is $H_{\CE}^1(L(F),L(F))\cong H_{\Harr}^1(A(F),A(F))$. The action of $\pi_1\Baut(F)$ on $\pi_n\Baut(F), n>1$  corresponds to the
adjoint action of the Lie algebra $H^1_{\CE}(L(F),L(F))$ on
$H^{1-n}_{\CE}(L(F),L(F))$ or, equivalently, to the adjoint action of
the Lie algebra $H^1_{\Harr}(A(F),A(F))$ on $H^{1-n}_{\Harr}(A(F),A(F))$.
\item
The group $\pi_1\Baut_*(F)\cong \Haut_*(F))$ is a group of $\mathbb Q$-points of an affine algebraic group over $\mathbb Q$ whose Lie algebra is $\overline{H}_{\CE}^1(L(F),L(F))\cong \overline{H}_{\Harr}^1(A(F),A(F))$. The action of $\pi_1\Baut_*(F)$ on $\pi_n\Baut_*(F), n>1$  corresponds to the adjoint action of the Lie algebra $\overline{H}^1_{\CE}(L(F),L(F))$ on $\overline{H}^{1-n}{\CE}(L(F),L(F))$ or, equivalently, to the adjoint action of the Lie algebra $\overline{H}^1_{\Harr}(A(F),A(F))$ on $\overline{H}^{1-n}_{\Harr}(A(F),A(F))$.
\end{enumerate}
\end{enumerate}
\end{theorem}
\begin{proof}
Note, first of all, that the various statements of the theorem come in pairs: one involving the CE complexes, the other the Harrison complexes. These statements are all pairwise equivalent by virtue of Theorem \ref{CEH} and we will give proofs for the CE versions only. Additionally, every statement has a pointed and an unpointed version; we will explain the latter in detail and indicate how the arguments should be modified for the former.

We start with (1). Let $V$ be a cofibrant simply-connected dgla  and consider the functor $V\mapsto\Ext_{L_\infty}(V,L(F))$.  Note that this functor is homotopy invariant in the sense that replacing $V$ by a weakly equivalent $V^\prime$ results in a natural bijection $\Ext_{L_\infty}(V,L(F))\cong \Ext_{L_\infty}(V^\prime,L(F))$. Therefore this functor lifts to the homotopy category of simply-connected dglas. Recall that the category of simply-connected dglas and that of simply-connected spaces are Quillen equivalent; denote by $Q(\g)$ the space corresponding to a dgla $\g$. Recall that by Proposition \ref{fibseq} an $L_\infty$ extension of dglas represents a homotopy fiber sequence of dglas.

Since the functor $Q$ takes a homotopy fiber sequence of dglas into a homotopy fiber sequence of pointed spaces and since the notion of equivalence of $L_\infty$ extensions is the same as the homotopy equivalence of the corresponding homotopy fiber sequences, we conclude that there is a natural bijection
\[
\Ext_{L_\infty}(V,L(F))\cong\Fib(Q(V),QL(F))\cong\Fib(Q(V),F).
\]
Now the functor $V\mapsto\Ext_{L_\infty}(V,L(F))$ is represented in the homotopy category of simply-connected dglas by the dgla $C_{\CE}(L(F),L(F))\langle 1\rangle$. This is because
\begin{align*}\Ext_{L_\infty}(V,L(F))&\cong [V,\Sigma C_{\CE}(L(F),L(F))]_{L_\infty}\\
&\cong [V,\Sigma C_{\CE}(L(F),L(F))\langle 1\rangle]_{L_\infty}
\end{align*}
where the last bijection is implied by Proposition \ref{conn}. Similarly the functor $Y\to\Fib(Y,F)$ is represented by $\Baut(F)$ and so we have the following bijections of sets of homotopy classes of \emph{pointed} maps:
\begin{align*}\Fib(Q(V),F)&\cong [Q(V),\Baut(F)]_*\\
&\cong [Q(V),\Baut(F)\langle 1\rangle]_*
\end{align*}
where the second bijection holds because $Q(V)$ is a simply-connected space. Since $Q$ is an equivalence of homotopy categories, it follows that the spaces $\Baut(F)\langle 1\rangle$ and $Q[\Sigma C_{\CE}(L(F),L(F))\langle 1\rangle]$ are homotopy equivalent, i.e. that $\Sigma C_{\CE}(L(F),L(F))\langle 1\rangle$ is a Lie-Quillen model for $\Baut(F)\langle 1\rangle$. This finishes the proof of (1).

The proof of (2) is similar, after one takes into account that $\Baut_*(F)$ classifies rooted fibrations having fiber $F$ and a \emph{section} and that $\Sigma \overline{C}_{\CE}(V,V)$ likewise classifies \emph{split} $L_\infty$ extensions with kernel $V$, cf. Theorem \ref{classs1}.

Let us now turn to part (3a), particularly its CE version. The finiteness assumptions on $F$ ensures that $\Haut(F)$ is the group of $\mathbb Q$-points of an affine algebraic group over $\mathbb Q$, see \cite{Sul, BL}. Moreover, the Lie algebra of this algebraic group has been identified in \cite{BL} in terms of the Harrison cohomology. The grading convention in op. cit differs with our present one  by a shift; after taking this into account we obtain the desired identification of the Lie algebra of $\Haut(F)$ in terms of the Harrison, (and thus, CE) cohomology.

Let $\alpha\in \pi_1\Baut(F)=\Haut(F)$; we choose its representative in $\Aut(F)$ and denote it by the same symbol. We interpret the adjoint action of $\alpha$ on $\pi_{n-1}\Aut(F)\cong \pi_n\Baut(F), n>1$ as follows. Consider the map $f_\alpha:\Aut(F)\to\Aut(F)$ given by $f_\alpha:\beta\to\alpha\beta\alpha^{-1}$ where $\beta\in\Aut(F)$ (here we use the fact that $\Aut(F)$ is a topological group, in particular that it possesses strict inverses).
Note that the map $f_\alpha$ is a topological group endomorphism and so it induces a self-map of the classifying space $\Baut(F)$, which we also denote by $f_\alpha$ by abuse of notation. This self-map depends on the choice of $\alpha$ within its homotopy class and homotopic choices give rise to homotopic self-maps of $\Baut(F)$. The homotopy class of $f_\alpha$ is, therefore, well-defined. It is clear that the induced map on the homotopy groups: $\pi_n\Baut(F)\to \pi_n\Baut(F)$ is just the adjoint action of $\alpha$ in the Whitehead Lie algebra $\pi_*\Baut(F)$.

Since $\Baut(F)$ is a classifying space for rooted fibrations with fiber $F$, the datum of a self-map of $\Baut(F)$ up to homotopy is equivalent to a natural transformation of the functor $\Fib(-,F)$. Let us identify this transformation.

Consider the commutative diagram of spaces:
\[
\xymatrix{
\Aut_*(F)\ar[d]\ar[r]&\Aut(F)\ar^{f_\alpha}[d]\ar[r]&F\ar^\alpha[d]\\
\Aut_*(F)\ar[r]&\Aut(F)\ar[r]&F
}
\]
Here the horizontal lines are the homotopy fiber sequences of the evaluation fibration $\Aut(F)\to F$ and the unmarked self-map of $\Aut_*(F)$ is induced by $f_\alpha$. Since $\Aut_*(F)\to\Aut(F)$ is a monoid map, the sequence can be delooped giving rise to the commutative diagram of homotopy fiber sequences
\[
\xymatrix{
F\ar^\alpha[d]\ar[r]&\Baut_*(F)\ar[d]\ar[r]&\Baut(F)\ar^{f_\alpha}[d]\\
F\ar[r]&\Baut_*(F)\ar[r]&\Baut(F)
}
\]
We will view the lower row of the above diagram as the universal rooted fibration over $\Baut(F)$ with fiber $F$ over the base point of $\Baut(F)$. It follows that the rooted fibration induced from the universal one by the map $f_\alpha$ is the \emph{same} universal fibration, but where the fixed homotopy equivalence of the central fiber with $F$ has been \emph{twisted} by a self-map $\alpha$ of $F$. This is how $\alpha\in \Haut(F)$ acts on the functor $\Fib(-,F)$: it leaves the  fibration unchanged but twists the given homotopy equivalence of the central fiber with $F$ by $\alpha$.

Next, let us consider the action of $\pi_1\Baut(F)=\Haut(F)$ on the functor $\Ext_{L_\infty}(-,L(F))$. First of all, replace $L(F)$ by an $L_\infty$-quasi-isomorphic minimal $L_\infty$ algebra $I$. Note that its representing complete cdga $\hat{S}\Sigma^{-1} I^*$ is nothing but a minimal Sullivan model for $F$; moreover the completion is unnecessary in this case since the grading considerations give $\hat{S}\Sigma^{-1} I^*\cong{S}\Sigma^{-1} I^*$. For an $L_\infty$ algebra $U$, the set of $L_\infty$ extensions $I\to V\to U$ is in 1-1 correspondence with the set $\MC(\Sigma C_{\CE}(I,I),(\hat{S}\Sigma^{-1}U^*)_+)$ where $\hat{S}\Sigma^{-1}U^*$ is the representing complete cdga of $U$. The group $\Haut(F)$ can be identified with the group of dg automorphisms of ${S}\Sigma^{-1} I^*$ (also known as \emph{curved} $L_\infty$ automorphisms of $I$) modulo automorphisms homotopic to zero.

Let again $\alpha\in \Haut(F)$ and let $\alpha$ be represented by a curved $L_\infty$ automorphism of $I$; it will be denoted by the same letter.
Thus, $\alpha\in\Aut(\hat{S}\Sigma^{-1}I^*)$. Then $\alpha$ determines an automorphism of the functor $\Ext_{L_\infty}(-,I)\cong[-,\Sigma C_{\CE}(I,I)]_{L_\infty}$ through its action by conjugation on $\Sigma C_{\CE}(I,I)=\Der(\hat{S}\Sigma^{-1}I^*)$. Replacing $\alpha$ by a homotopic curved $L_\infty$ automorphism results
in a homotopic self-map of $\Sigma C_{\CE}(I,I)$. There is, therefore, a map of monoids $\Haut(F)\to [\Sigma C_{\CE}(I,I),\Sigma C_{\CE}(I,I)]_{L_\infty}$ or, in other words, an action (up to homotopy) of the group $\Haut(F)$ on $\Sigma C_{\CE}(I,I)\cong \Sigma C_{\CE}(L(F),L(F))$ and thus, also on $\Sigma C_{\CE}(L(F),L(F))\langle 1\rangle$.

The associated (honest) group action of $\Haut(F)$ on $H_{\CE}(L(F),L(F))\langle 1\rangle$ coincides with the topological action of $\pi_1(\Baut(F))$ on $\pi_n(\Baut(F)), n>1$ through the identification of $\Haut(F)$ with $\pi_1(\Baut(F))$ and of $\pi_n(\Baut(F))$ with $H_{\CE}^{1-n}(L(F),L(F))$.

The tangent Lie algebra of the quotient of the group of all $L_\infty$ automorphisms of $I$ by those homotopic to zero is precisely $H_{\CE}^1(I,I)$; this fact (more precisely, the equivalent statement involving Harrison cohomology) was established in \cite{BL}. It follows that the tangent action of $H_{\CE}^1(I,I)\cong H_{\CE}^1(L(F),L(F))$ on  $\pi_n(\Baut(F))\cong H_{\CE}^{1-n}(L(F),L(F))$ is via the commutator bracket in $H_{\CE}^{*}(L(F),L(F))$ as claimed. This finishes the proof of (3a).

The proof of (3b) is similar. We first interpret the action of an element in $\Haut_*(F)\cong\pi_1\Baut_*(F)$ in terms of natural transformations of the functor $\Fib_*(-,F)$ associating to a topological space $Y$ the set of equivalence classes of \emph{sectioned} fibrations over $Y$ with
fiber $F$. We next observe that the group $\Haut_*(F)$ is the set of $\mathbb Q$ points of an algebraic group and the corresponding Lie algebra
is $\overline{H}^1_{\CE}(L(F),L(F))$. The group of $L_\infty$ automorphisms of $I$ acts by conjugations on $\Sigma \overline{C}_{\CE}(I,I)\cong \Sigma\overline{C}_{\CE}(L(F,L(F))$ and this determines an action up to homotopy of the group $\Haut_*(F)$  on $\Sigma\overline{C}_{\CE}(L(F),L(F))$.

The associated (honest) action of $\Haut_*(F)$ on $\overline{H}^{1-n}_{\CE}(L(F),L(F))$ coincide with the action of $\pi_1\Baut_*(F)$ on $\pi_n\Baut_*(F)$ for $n>1$ and the tangent action of $ H_{\CE}^1(L(F),L(F))$ on  $\cong H_{\CE}^{1-n}(L(F),L(F))$ is via the commutator bracket in $H_{\CE}^{*}(L(F),L(F))$.
\end{proof}
It follows that the Whitehead Lie algebras of $\Baut(F)$ and of $\Baut_*(F)$ can be computed solely in terms of standard derived functors.
\begin{cor}\
\begin{enumerate}
\item
The Whitehead Lie algebra $\pi_*\Baut(F)\langle 1\rangle$ is isomorphic to either of the graded Lie algebras $\oplus_{n=2}^{-\infty}H_{\CE}^n(L(F),L(F))\cong \oplus_{n=2}^{-\infty}H_{\Harr}^n(A(F),A(F))$.
\item
The Whitehead Lie algebra $\pi_*\Baut_*(F)\langle 1\rangle$ is isomorphic to either of the graded Lie algebras $\oplus_{n=2}^{-\infty}\overline{H}_{\CE}^n(L(F),L(F))\cong \oplus_{n=2}^{-\infty}\overline{H}_{\Harr}^n(A(F),A(F))$.
\end{enumerate}
\end{cor}\noproof
Furthermore, since a Sullivan model (at least for simply-connected spaces) is obtained from its Lie-Quillen model by applying the functor $C_{\CE}$ we get the following result.
\begin{cor}
Either of the cdgas $C_{\CE}\big(C_{\CE}(L(F,L(F))\langle 1\rangle \big)$ or $C_{\CE}\big(C_{\Harr}(A(F),A(F))\langle 1\rangle\big)$ is a Sullivan model (non-minimal in general) for the space $\Baut(F)\langle 1\rangle$.
\end{cor}
\noproof

Standard examples of explicit models for $\Baut(F)$ include the cases when $F$ is a complex projective space or a wedge of spheres, cf. \cite{Sul, SS}. The following example is less well-known.
\begin{example}
Consider the wedge of $2N$ spheres of the form $X=\bigvee_{1}^{2N}S^n$ where $n>1$; we also assume that $n$ is odd. Then $X$ has a Lie-Quillen model of the form ${\mathbb L}\langle p_1, q_1,\ldots, p_N,q_N\rangle$, the free Lie algebra on $2N$ generators $p_i, q_i$ with $|p_i|=|q_i|=n-1$.  We could build a Poincar\'e duality space $M$ out of $X$ by attaching a $2n$-cell corresponding to the element $w=[p_1,q_1]+\ldots+[p_N,q_N]\in\pi_n(X)$. Then $M$ is a coformal space whose Lie-Quillen model is $L(M):={\mathbb L}\langle p_1, q_1,\ldots, p_N,q_N\rangle/\omega$.

Let us forget for a moment the internal grading on $L(M)$ and view it as an ungraded Lie algebra. Then its universal enveloping algebra $U\big(L(M)\big)$ is a preprojective algebra corresponding to a quiver having one vertex and $n$ loops. The Hochschild cohomology of preprojective algebras was computed in \cite{CEG} where it was proved, in particular,  that it vanishes in degrees greater than $2$. It follows that $H^n_{\CE}(L(M),L(M))=0$ for $n>2$. Further, $H^2_{\CE}\big(L(M),L(M)\big)$ splits off $H^2_{\Hoch}\big(U(L(M)),U(L(M))\big)$ and it follows from the explicit form of $H^2_{\Hoch}\big(U(L(M)),U(L(M))\big)$ computed in op. cit. that $H^2_{\CE}(L(M),L(M))$ is a 2n-dimensional vector space concentrated in homological degree $-n$; denote it by $H$. We have therefore an isomorphism $\overline{H}^*_{\CE}(L(M),L(M))\cong \Der(L(M),L(M))\ltimes H$. The Lie algebra of derivations of $L(M)$ is just the Lie subalgebra $\g_N$ of derivations of the free Lie algebra ${\mathbb L}\langle p_1, q_1,\ldots, p_N,q_N\rangle$ which preserve the ideal generated by $w$\footnote{The author would like to thank A. Berglund for pointing this out to him.}. Note that $\g_N$ contains a Lie subalgebra, consisting of derivations \emph{vanishing on $w$}; the latter Lie algebra was introduced by Kontsevich, \cite{Kon}; the CE complex of its stable version computes, essentially, the cohomology of groups of outer automorphisms of free groups.

Remembering that $L(M)$ had a grading, we see that $\g_N\ltimes H$ also has one and thus, we can take its simply-connected cover $(\g_N\ltimes H)\langle 1\rangle\cong \g_N\langle 1\rangle$. This is a Lie-Quillen model of $\Baut_*(M)\langle 1\rangle$. In particular, this space is coformal\footnote{These observations were also made by A. Berglund and I. Madsen, \cite{MB}, using different methods.}.

It is also curious to note that the $S^1$-equivariant homology free loop space on $X$ supports a string bracket which is essentially the Kontsevich noncommutative Poisson bracket \cite{Kon} on $\Der\big(U(L(M))\big)$; this was observed in \cite{Lazs}.
\end{example}
\section{Algebraic structure on the deformation complexes of $L_\infty$ algebras}
Suppose that $U, V$ are two $L_\infty$ algebras and $f:V\to U$ is an $L_\infty$ map represented by a map of complete cdgas $\hat{S}\Sigma U^*\to \hat{S}\Sigma^{-1} V^*$. Consider the $L_\infty$ algebra $\hat{S}\Sigma^{-1}  V^*\otimes U$, the tensor product of $U$ and the complete cdga $\hat{S}\Sigma^{-1}  V^*$ representing $U$; the map $f$ can be viewed as an MC element in this $L_\infty$ algebra. Then, as explained in \cite{Laz}, the CE complex of $V$ with coefficients in $U$ is the $L_\infty$ algebra $(\hat{S}\Sigma^{-1}  V^*\otimes U)^f$, the \emph{twisting} of $\hat{S}\Sigma^{-1}  V^*\otimes U$ by the MC element $f$. Incidentally, this is the $L_\infty$ algebra `governing' the deformations of the $L_\infty$ map $f$.

Now let $V=U$. As we know, the complex $C_{\CE}(V,V)$ admits a different description; namely $C_{\CE}(V,V)=\Sigma^{-1}\Der(\hat{S}\Sigma V^*)$. The commutator of derivations endows $\Sigma C_{\CE}(V,V)$ with the Gerstenhaber bracket, making it into a dgla. As was explained earlier in the paper, this dgla governs the deformations of the $L_\infty$ structure on $V$.

We see, therefore, that the complex $C_{\CE}(V,V)$ possesses two structures: the structure of an odd dgla with the Gerstenhaber bracket and that of an $L_\infty$ algebra. The latter structure is somewhat less familiar, even in the situation when $V$ is an ordinary Lie algebra (although it was considered in \cite{Mar, LM}). Note that in that case the $L_\infty$ structure on $C_{\CE}(V,V)$ reduces to that of a dgla; we will refer to the corresponding bracket as the cup-bracket. It could be described in the traditional notation as follows. Let $g\in C_{\CE}^n(V,V)$ and $h\in C_{\CE}^m(V,V)$ be two CE cochains of the Lie algebra $V$ viewed as skew-symmetric multilinear maps:
$g:\Lambda^n(V)\to V; h:\Lambda^m(V)\to V$. Then we can form their  cup-bracket $[g\cup h]:\Lambda^{n+m}(V)\to V$:
\[
 [g\cup h](v_1,\ldots, v_{m+n})=\sum_{\sigma\in S_{n+m}}\frac{1}{(n+m)!}(-1)^{\sigma}[g(v_{\sigma(1)}, \ldots, v_{\sigma(n)}),h(v_{\sigma(n+1)}, \ldots, v_{\sigma(n+m)})].
\]
In the $L_\infty$ context similar formulas can also be written down but it is much simpler to work in the dual framework, using representing
complete cdgas instead.

If the representing complete cdga of $V$ is a Sullivan model of a topological space $X$, then the Gerstenhaber bracket in $\Sigma C_{\CE}(V,V)$ corresponds to
the Whitehead Lie bracket on homotopy groups of $\Baut(X)$ whereas the cup-bracket corresponds to the Whitehead Lie bracket on homotopy groups of $\Aut(X)$. Since the latter is an H-space, the latter bracket is trivial.

This suggests that the cup-bracket on $C_{\CE}(V,V)$ for an arbitrary $L_\infty$ algebra $V$ should be trivial, at least up to homotopy. We will see that deformation theory provides a short and non-computational proof of this statement. A similar argument can also be used for other type of cohomology, i.e. Harrison or Hochschild but we will restrict ourselves to the CE case. We note that this argument, in the Hochschild context, could be found in Kontsevich's 1994 lectures on deformation theory \cite{Konlect}.
\subsection{Unobstructed deformation functors} The material in this subsection has been known to experts since Kontsevich's lectures on deformation theory \cite{Konlect}, however as far as we know it has not been properly documented and we feel that it is useful to give a brief outline here.

\begin{defi} Let $(V,m)$ be an $L_\infty$ algebra and $A$ be a non-unital complete cdga. Then an element $\xi\in (A\otimes\Sigma V)_0$ is
\emph{Maurer-Cartan} (MC for short) if $(d_A\otimes\id)(\xi)+(\id\otimes d_V)+\sum_{i=1}^\infty
\frac{1}{i!}m^A_i(\xi^{\otimes i})=0$. The set of Maurer-Cartan elements in $A\otimes\Sigma V$ will be
denoted by $\MC(V,A)$.

Two MC elements $\xi,\eta\in\MC(V,A)$ are called \emph{homotopic} if there exists
an MC-element $h\in(V, A[z,dz])$ such that $h|_{z=0}=\xi$ and $h|_{z=1}=\eta$.
The set of homotopy classes of MC elements in $\MC(V,A)$ will be denoted by $\MCmod(V,A)$.
\end{defi}

An MC element $\xi\in\MC(V,A)$ can be viewed as a \emph{deformation} of the zero MC element with base $A$ (or $\tilde{A}$ if one wishes to work with unital bases). The functor $A\mapsto\MCmod(V,A)$ is also called the (extended) \emph{deformation functor} associated with an $L_\infty$ algebra $V$ and having $A$ as the base. To reinforce this point of view we will use the notation $\Def_V$ for this functor. It follows from Hinich's results \cite{H} that this $\Def_V$ is representable in the homotopy category of complete cdgas by $\hat{S}\Sigma^{-1}V^*$, the representing cdga of $V$ (cf. \cite{Laz} which explains this).
\begin{defi}\
\begin{enumerate}
\item
A (non-unital) complete cdga $A$ is called \emph{infinitesimal} if its multiplication is zero.
\item
A deformation with an infinitesimal base is called an \emph{infinitesimal deformation}.
\item
The deformation functor associated with an $L_\infty$ algebra $V$ is called \emph{unobstructed} if for any surjective map of complete cdgas $f:A\to B$, both having zero differential, such that the kernel of $f$ is infinitesimal, the corresponding map $\Def_V(A)\to\Def_V(B)$ is surjective.
\end{enumerate}
\end{defi}
We now have the following almost obvious result.
\begin{theorem}
Let $V$ be an $L_\infty$ algebra. Then the functor $\Def_V$ is unobstructed if and only if $V$ is $L_\infty$ quasi-isomorphic to an abelian $L_\infty$ algebra, i.e. an $L_\infty$ algebra whose $L_\infty$ products are all zero. In such a situation we will call $V$ \emph{homotopy abelian}.
\end{theorem}
\begin{proof}
Let $V$ be homotopy abelian. Since the functor $\Def_V$ is homotopy invariant in $V$ we might as well assume that $V$ is abelian, i.e. that its representing cdga is $(\hat{S}\Sigma^{-1}V^*,0)$, with \emph{zero differential}. The problem of lifting a deformation over $B$ to a deformation over $A$ is the problem of constructing the dotted arrow in the homotopy category of complete (non-unital) cdgas:
\[
\xymatrix{
&A\ar[d]\\
(\hat{S}\Sigma^{-1}V^*)_+\ar[r]\ar@{-->}[ur]&B
}
\]
which is always possible since $(\hat{S}\Sigma^{-1}V^*)_+$ is a free object in complete cdgas. Conversely, suppose that a lifting is always possible. Then clearly for any infinitesimal algebra $B$ any deformation over $B$ lifts to a deformation over the formal power series ring $\ground[[B]]$. Without loss of generality we can assume that the $L_\infty$ algebra $V$ is minimal, i.e. that the differential on $V$ is zero. Consider the universal infinitesimal deformation with base $\Sigma^{-1}V^*$ given by the quotient map $((\hat{S}\Sigma^{-1}V^*)_+,m_V)\to \Sigma^{-1}V^*$. The lift of this deformation to the formal power series ring determines an isomorphism of the representing complete cdga of $V$ with $((\hat{S}\Sigma^{-1}V^*)_+,0)$, i.e. a formal power series ring with \emph{zero differential}. Thus, $V$ is homotopy abelian.
\end{proof}
It is, of course, very rare for a deformation problem to be unobstructed; however there are important cases when it happens. For example, if a dgla $V$ supports a dg BV-algebra structure of a special kind \cite{BK}, then the associated deformation problem is unobstructed; in particular deformations of Calabi-Yau manifolds are unobstructed. Similarly, deformations of a symplectic structure on a symplectic manifold $M$ are unobstructed since the governing dgla is the de Rham complex of $M$ with the trivial bracket. We will now discuss another example of this kind and its consequences for the CE cohomology.
\subsection{Deformations of automorphisms of $L_\infty$ algebras} Let $U$ and $V$ be two $L_\infty$ algebras and $f:V\to U$ be an $L_\infty$ map. Such a map can be viewed as an MC element in the $L_\infty$ algebra $\overline{L}:=(\hat{S}\Sigma^{-1}U^*)_+\otimes V$. Consider the deformation functor $\Def_{\overline{L}^f}$ associated with $\overline{L}$ twisted by $f$. One can say that deformations of $f$ are governed by the $L_\infty$ algebra $\overline{L}^f$. Consider also $L:=\hat{S}\Sigma^{-1}U^*\otimes V$ and the $L_\infty$ algebra $L^f$. The algebra $L^f$ governs deformations of $f$ within the class of \emph{curved} $L_\infty$ maps $V\to U$.

Now assume that $U=V$ and that $f$ is the identity morphism. In that case it is known that $L^f\cong\overline{C}_{\CE}(V,V)$ and $L^f\cong{C}_{\CE}(V,V)$, cf. \cite{Laz}.
\begin{theorem}\label{weakD}
The $L_\infty$ algebras $C_{\CE}(V,V)$ and $\overline{C}_{\CE}(V,V)$ are both homotopy abelian.
\end{theorem}
\begin{proof}
We will start with the $L_\infty$ algebra $L^{\id}=C_{\CE}(V,V)$ and consider the functor $\Def_{L^{\id}}$. Let $A$ be an infinitesimal complete algebra; then a deformation $\xi$ associated to $L^{\id}$ with base $A$ can be identified with a continuous map of cdgas $\hat{S}\Sigma^{-1}V^*\to \tilde{A}\otimes\hat{S}\Sigma^{-1}V^*$ whose composition with the projections $\tilde{A}\otimes\hat{S}\Sigma^{-1}V^*\to \hat{S}\Sigma^{-1}V^*$ is the identity map. Then $\xi$ is nothing but a (continuous) $A$-linear derivation of $\tilde{A}\otimes\hat{S}\Sigma^{-1}V^*$ whose image lies in ${A}\otimes\hat{S}\Sigma^{-1}V^*$. Furthermore, $\xi$ can also be viewed as an $\hat{S}(A)$-linear derivation of $\tilde{A}\otimes\hat{S}\Sigma^{-1}V^*$ whose image lies in the ideal generated by $A$. It follows that $e^\xi$ is an $\hat{S}(A)$-linear automorphism of $\tilde{A}\otimes\hat{S}\Sigma^{-1}V^*$ whose linear part is $\xi$. Thus, the deformation functor associated with the $L_\infty$ algebra $C_{\CE}(V,V)$ is unobstructed and the latter is homotopy abelian.

The proof for $\overline{C}_{\CE}(V,V)$ is similar except $\hat{S}\Sigma^{-1}V^*$ has to be replaced with its non-unital version $(\hat{S}\Sigma^{-1}V^*)_+$.
\end{proof}
\begin{rem}
The proof above can be summed up by saying that deformations of automorphisms are unobstructed since every infinitesimal automorphism can be exponentiated to a formal automorphism. Note that the ensuing $L_\infty$ isomorphism with an abelian $L_\infty$ algebra is by no means trivial. For example, if $V$ is an ordinary Lie algebra the CE complex $C_{\CE}(V,V)$ is isomorphic, as a graded Lie algebra to the tensor product $\hat{S}\Sigma^{-1}V^*\otimes V$. This dgla becomes homotopy abelian only after twisting by the canonical MC element (which, in this case, does not change the graded Lie algebra structure).
\end{rem}
\section{Future directions}
Our treatment of rational homotopy theory as deformation theory gives rise to a number of natural open questions and possible applications. Following the suggestion of the referee, we gathered them in the concluding section of the paper.

\subsection{Deformation theory}Theorem \ref{defo} can colloquially be rephrased as saying that the dgla $\Sigma\overline{C}_{\CE}(I,I)$ supplied with the Gerstenhaber bracket \emph{governs} deformations of the $L_\infty$ algebra $I$. This slogan is, of course, well-known.

In this connection it would be interesting to make a comparison between our results and Hinich's, \cite{H'}. He obtains a similar result (in the generality of operadic algebras) but \emph{only} under the assumption that both the algebra being deformed and the base are cohomologically non-positively graded whereas we don't have such restrictions. Let $V$ be an (operadic) algebra and $A$ be the base of a deformation, i.e. a complete cdga (or a dg artinian algebra). Hinich assumes that $V$ is cofibrant and also, that deformations of $V$ are cofibrant as $A$-algebras. It appears that it is the latter condition that leads to the discrepancy between our approaches. Presumably, there is a closed model category structure on $A$-algebras which takes into account the filtration on $A$ and such that, in particular, $A$-algebras having zero homology are not necessarily weakly equivalent to zero. It seems likely that Hinich's approach applied to this notion of deformation allows one to get rid of grading assumptions.

Next, let $I$ be an $A_\infty$ algebra. Then one can define the notion of a deformation of $I$ over a complete cdga $A$ as well as the equivalence of such deformations. Such a notion is essentially equivalent to Kajiura-Stasheff's \emph{open-closed homotopy algebra} (OCHA), \cite{KS}. One can also prove that the functor of deformations is represented by the complete cdga $C_{\CE}\big(\Sigma\overline{C}_{\Hoch}(I,I)\big)$ where $\Sigma\overline{C}_{\Hoch}(I,I)$ is the truncated Hochschild complex of $I$ with the Gerstenhaber bracket, cf. for example \cite{CL} concerning this notion. However, an interpretation of such a deformation as an extension is lacking, essentially because we still consider commutative (albeit dg and complete) bases. The relevant deformation theory, allowing noncommutative bases, was considered in the recent paper \cite{GS}. It seems likely that using the methods of op.cit. one could prove representability of the deformation functor in the appropriate homotopy category and  a similar classification result in the noncommutative situation.  It will also be interesting to extend this theory to algebras over other cofibrant operads (e.g. $C_\infty$ and $G_\infty$).
\subsection{Rational homotopy}
Let $F$ be as in Section 5. Theorem \ref{mainth} provides a nearly complete algebraic description of the classifying spaces of $\Aut(F)$ and $\Aut_*(F)$. More precisely,  the Harrison or Chevalley-Eilenberg models allow one to reconstruct universal covers of these spaces together with the action of the fundamental 
group on higher homotopy groups by conjugation. In order to reconstruct the full classifying space, say, of $\Aut(F)$, one has to know the corresponding action of $\pi_1\Baut(F)$ on the corresponding (Harrison or Chevalley) model of $\Baut(F)\langle1\rangle$. This action is, of course, not the adjoint action mentioned above because topologically $\pi_1\Baut(F)$ acts on $\Baut(F)\langle1\rangle$ without fixed points. This is, therefore, an \emph{affine} action and the natural question is to identify it, or the tangent action of the corresponding Lie algebra in terms of Chevalley-Eilenberg or Harrison complexes. 

Next, let $I$ be a minimal $L_\infty$ model of $F$ (so that the representing cdga of $I$ is a Sullivan minimal model of $F$). Since a minimal model of $I$ is is unique up to an isomorphism, the group of all (curved) $L_\infty$ automorphisms of a minimal model of $I$ has a homotopy invariant meaning, as well as does its adjoint action on the dgla $\Sigma C_{\CE}(I,I)$. Can one describe this group and its action in intrinsic algebro-topological terms and if so is this sufficient to reconstruct the (homotopy type of) $\Baut(F)$ or $\Baut_*(F)$?

Finally, suppose that $F$ is \emph{not} a nilpotent space. In this case there is still an $L_\infty$ minimal model of $F$ which is unique up to an isomorphism. It follows that the dgla $\Sigma C_{\CE}(I,I)$ is a homotopy invariant of $F$. Therefore, so is its simply-connected truncation $\Sigma C_{\CE}(I,I)\langle 1\rangle$. Then the simplicial set $F^\prime:=Q\left(\Sigma C_{\CE}(I,I)\langle 1\rangle\right)$ also depends on the homotopy type of $F$ only. Since $F$ is not nilpotent, we cannot conclude that $F^\prime$ is (rationally homotopy equivalent to) the universal cover of $F$. Can one still say anything about $F^\prime$ in terms of $F$? 
\subsection{Deligne conjecture type problems}
It is clear that the argument used in the proof of Theorem \ref{weakD}  can be applied to the deformation complexes of other homotopy algebras. For instance, if $V$ is an $A_\infty$ algebra then its Hochschild complex $C_{\Hoch}(V,V)$ has the structure of an $A_\infty$ algebra and one can similarly prove that its associated $L_\infty$ algebra is homotopy abelian, in particular, that the associative structure on the cohomology is commutative. Note that the $A_\infty$ structure on $C_{\Hoch}(V,V)$ is, in fact, equivalent to a $C_\infty$ algebra structure as part of a richer structure of a $G_\infty$ algebra implied by the Deligne conjecture. This is, of course, stronger than the statement that $C_{\Hoch}(V,V)$ is homotopy abelian as an $L_\infty$ algebra, thus our statement can be viewed as a weak version of the Deligne conjecture for $L_\infty$ algebras.
This makes one wonder if a stronger version of this conjecture is true for $L_\infty$ (or even for Lie) algebras. There is a notion of the operad of natural operations in the CE complex $C_{\CE}(V,V)$ of an $L_\infty$ algebra $V$ (indeed, such a notion can be defined in a much more general situation, including Hochschild and Harrison cohomology and many more, \cite{Mar}). The Gerstenhaber bracket and $L_\infty$ products are examples of natural operations; these generate a dg suboperad inside the dg operad of all natural operations in $C_{\CE}(V,V)$. A natural question is to find a presentation of this operad by generators and relations; that means effectively finding a compatibility condition between the $L_\infty$ operations and the Gerstenhaber bracket. A similar question could be asked about the Harrison complex  $C_{\Harr}(A,A)$ where $A$ is a commutative or, more generally, $C_\infty$ algebra. Theorem \ref{CEH} suggests that the answer for the commutative and Lie cases should be the same. Furthermore, from Theorem \ref{weakD} one could expect that that this suboperad is quasi-isomorphic to the operad of Lie algebras. This could be viewed as a version of the Deligne conjectire for $L_\infty$ algebras and we refer the reader to the above-mentioned paper by Markl where many more related conjectures are formulated.

\end{document}